\documentclass[11pt, reqno]{amsart}
\usepackage{graphicx} 
\usepackage{amsmath,amssymb,amscd,mathrsfs,amscd}
\usepackage{graphics,verbatim, tikz}

\usepackage{hyperref}
\usepackage{pifont}
\usepackage{mdframed} 
\usepackage{enumerate}

\usepackage{graphicx}
\usepackage{epstopdf}
\DeclareMathAlphabet{\mathpzc}{OT1}{pzc}{m}{it}
\usepackage{tikz,tikz-cd, color}
\usepackage[makeroom]{cancel}
\usetikzlibrary{arrows}

\usepackage{mathdots}
\usepackage{tikz}
\usepackage{quantikz}
\usepackage{algorithm2e}
\usepackage{float}
\usepackage{caption, subcaption}
\usepackage{rotating}
\usepackage{pdfpages}
\allowdisplaybreaks

\usepackage{adjustbox}
\usetikzlibrary{matrix}

\tikzset{
  symbol/.style={
    draw=none,
    every to/.append style={
      edge node={node [sloped, allow upside down, auto=false]{$#1$}}}
  }
}

\usetikzlibrary{decorations.shapes}

\tikzset{decorate sep/.style 2 args=
{decorate,decoration={shape backgrounds,shape=circle,shape size=#1,shape sep=#2}}}

\newtheorem{theorem}{Theorem}[section]
\newtheorem{lemma}[theorem]{Lemma}

\newtheorem{corollary}[theorem]{Corollary}

\newtheorem{definition}[theorem]{Definition}

\newtheorem{remark}[theorem]{Remark}

\theoremstyle{definition}

\newtheorem*{theorem*}{Theorem}

\newcommand{\cC}{\mathcal{C}}

\newcommand{\cG}{\mathcal{G}}

\newcommand{\cI}{\mathcal{I}}

\newcommand{\cN}{\mathcal{N}}

\newcommand{\cP}{\mathcal{P}}

\newcommand{\up}[2]{#1^{(#2)}}
\newcommand{\cent}{{\rm Cent}}
\newcommand{\geng}[3]{\left\langle
\begin{matrix}
#1 ,\\
#2 ,\\
#3
\end{matrix}\right\rangle}

\newcommand{\nf}{{\rm NF}}
\newcommand{\cli}{{\rm Cliff}}

\begin{document}
\title[Projective Clifford Groups]{Embedding the $n$-Qubit Projective Clifford Group into a Symmetric Group}
\author[Lee]{Chin-Yen Lee}
\address{Department of Mathematics, National Tsing Hua University, Hsinchu City, 30013, Taiwan.}
\email{cyl@math.nthu.edu.tw}

\begin{abstract}
In this paper, we construct a symmetric group ${\rm Sym}_{2(4^n-1)}$, which contains a subgroup isomorphic to the $n$-qubit projective Clifford group $\mathcal{C}_n$. To establish this result, we investigate the centralizers of the $z$ gate and the phase gate within the $n$-qubit projective Clifford group, utilizing the normal form of the Clifford operators. As a byproduct, we also provide a presentation of the inertia subgroup of $\mathcal{C}_n$.
\end{abstract}


\maketitle

\section{Introduction}\label{sec:intro}

In quantum computing, the qubit serves the same role as the bit in classical computing and is represented as a two-dimensional vector space. Operations on qubits are described by $2 \times 2$ unitary matrices. These operations are also called gates, as they are used in quantum circuits.
The main difference between a qubit and a bit is that when we connect qubits, we take their tensor product. The same applies to gates.

Here are some common gates used in quantum computing:
\begin{align*}
	& Z=\begin{bmatrix} 1&0\\0&-1\end{bmatrix},\quad    X=\begin{bmatrix} 0&1\\1&0\end{bmatrix},\quad H=\frac 1 {\sqrt 2}\begin{bmatrix}1&1\\1&-1\end{bmatrix},\quad S=\begin{bmatrix} 1&0\\0&\omega^2\end{bmatrix},\quad I_2=\begin{bmatrix}1&0\\0&1\end{bmatrix},\\
	&    CZ=\begin{bmatrix} 1&0&0&0\\ 0&1&0&0\\0&0&1&0\\0&0&0&-1\end{bmatrix},\quad CX=\begin{bmatrix} 1&0&0&0\\ 0&1&0&0\\0&0&0&1\\0&0&1&0\end{bmatrix},\quad   SWAP=\begin{bmatrix} 1&0&0&0\\ 0&0&1&0\\ 0&1&0&0\\0&0&0&1\end{bmatrix}
\end{align*}
where $\omega$ is a primitive $8$th root of $1$.
The $H$ gate is known as the Hadamard gate, while the $S$ gate is known as the phase gate.
The $CZ$ (controlled $Z$) gate and the $CX$ (controlled $X$) gate have qubit $1$ as the control and qubit $2$ as the target.  The $SWAP$ gate exchanges the $1$st and $2$nd qubits. 
We refer the reader to \cite{NC} for a complete introduction.
When considering gates on $n$-qubits, we can apply the above gates to specific qubits. For example, 
\begin{align*}
	U_i&:= I_2^{\otimes (i-1)}\otimes U \otimes I^{\otimes n-i}_2\quad\text{for } U\in \{Z,X,H,S\},i\in\{1,\dots,n\}
\end{align*}
are the corresponding gates that act only on the $i$th qubit. The controlled $Z$ gate and the controlled $X$ gate 
\begin{align*}
	CZ_{i\to i+1}&:= I_2^{\otimes (i-1)}\otimes CZ \otimes I^{\otimes n-i-1}_2\quad\text{for }i \in\{1,\dots,n-1\},\\
	CX_{i\to i+1}&:= I_2^{\otimes (i-1)}\otimes CX \otimes I^{\otimes n-i-1}_2\quad\text{for }i \in\{1,\dots,n-1\}
\end{align*}
have qubit $i$ as the control and qubit $i+1$ as the target.
The gate
\begin{align*}
	SWAP_{(i,i+1)}&:= I_2^{\otimes (i-1)}\otimes SWAP \otimes I^{\otimes n-i-1}_2\quad\text{for }i \in\{1,\dots,n-1\}
\end{align*}
exchanges the $i$th and $(i+1)$th qubits.

The Pauli groups and the Clifford groups are well-known for their applications in quantum computing.
\begin{definition}
	The \emph{$n$-qubit Pauli group} ${\rm Pauli}_n$ is the subgroup of the $2^n\times 2^n$ unitary matrices  generated by $n$-fold tensor products of $X,Z,\omega^2I_2$:
	\begin{align*}{\rm Pauli}_n=
		\geng{X_1,\dots, X_n}{Z_1,\dots,Z_n}{\omega^2I_2^{\otimes n} }.
	\end{align*}
\end{definition}

\begin{definition}
	The $n$-qubit Clifford group ${\rm Cliff}_n$ is the normalizer of the $n$-qubit Pauli group in the unitary group:
	\begin{align*}
		{\rm Cliff}_n =\left\{ U\in U(2^n): UPU^{-1}\in  {\rm Pauli}_n\text{ for all }P\in {\rm Pauli}_n\right\}.
	\end{align*}
	
\end{definition}
In this paper, we restrict the global phases to $\mathbb{Q}[\omega]$ to ensure that $\cli_n$ becomes a finite group.

In quantum physics, we cannot distinguish between two quantum states that differ only by a global phase, as such a phase does not affect measurable quantities or probabilities. 
Hence, it is common to use their projective version:
\begin{align*}
	\cP_n={\rm Pauli}_n/\langle \omega^2 I_{2}^{\otimes n}\rangle,\quad\cC_n={\rm Cliff}_n/\langle \omega I_{2}^{\otimes n}\rangle.
\end{align*}

To facilitate tracking which group we are discussing, we will use letters $Z_i, X_i, H_i, S_i$, $CZ_{i \to i+1}$, and $SWAP_{(i,i+1)}$ for ${\rm Cliff}_n$, while using letters $z_i, x_i, h_i, s_i, \Lambda_i$, and $\sigma_{(i,i+1)}$ for $\mathcal{C}_n$. 
Despite this notation, both sets of elements are referred to as Clifford operators. 
Recall that the Hadamard gates, phase gates, and controlled $z$ gates generate $\cC_n$:
\begin{align*}
	\cC_1=\langle h_1,s_1\rangle,\quad	\cC_n=\geng{h_1,\dots,h_n}{s_1,\dots,s_n}{\Lambda_1,\dots,\Lambda_{n-1}}\quad\text{for }n\geq 2.
\end{align*}
A completed presentation of $\cC_n$ is recorded in Appendix \ref{sec:appendix}.
We consolidate the common gates among these generators. Here, the notation $[\cdot]:{\rm Cliff}_n\to \cC_n $ denotes the quotient that defines $\cC_n$:
\begin{align*}
	[Z_i]=s_i^2,\quad [X_i]=h_is_i^2h_i,\quad [CX_{i\to i+1}]=h_{i+1}z_ih_{i+1},\quad[SWAP_{(i,i+1)}]=(\Lambda_ih_{i+1}h_i)^3.
\end{align*}

Our goal is to study the structure and representation theory of $\mathcal{C}_n$. 
This is inconvenient; for example, constructing $\mathcal{C}_n$ on a computer using the original definition involves working with a quotient. 
However, there are several equivalent descriptions of $\mathcal{C}_n$. 
For example, its elements can be represented by binary tableaux over $\mathbb Z_2$ \cite{DD,AG,Van}. 
Alternatively, it can be depicted as a quantum circuit in a normal form \cite{Sel}, which will be introduced in the Section \ref{sec:nf}. 
This format provides a presentation for $\mathcal{C}_n$, which is included in Appendix \ref{sec:appendix}.
In this paper, we show that elements of $\mathcal{C}_n$ can be represented as permutations within the conjugacy class of $s_1$ in $\cC_n$.
\begin{theorem*}
	Let $V$ be the conjugacy class of $s_1$ in $\cC_n$ and let $\sigma_g$ denote the permutation induced by the conjugate action of $g \in \cC_n$ on $V$. Then the map $\cC_n \to {\rm Sym}_{|V|}: g \mapsto \sigma_g$ is a faithful homomorphism. Moreover, the size of $V$ is equal to $2(4^n - 1)$.
\end{theorem*}
Note that there is no effect of phase in the conjugacy action of ${\rm Cliff}_n$. 
Additionally, observe that ${\rm Trace}(\omega^k S_1) = {\rm Trace}(U S_1 U^{-1})={\rm Trace}(S_1)$ for any $U \in {\rm Cliff}_n$ only holds when $k \equiv 0$ (mod $8$).
Thus, the size of the conjugacy class of $S_1$ in ${\rm Cliff}_n$ is the same as the size of the conjugacy class of $s_1$ in $\cC_n$. It suffices to describe the corresponding permutations of the generators $H_i, S_i, CZ_{i \to i+1}$ of ${\rm Cliff}_n$ acting on the conjugacy class of $S_1$ in ${\rm Cliff}_n$.
Using this process to construct $\cC_n$ involves only matrix multiplication and does not require taking a quotient.
This result allows us to study $\cC_n$ and its character table for $n \leq 6$ more easily, for example, by using GAP4 \cite{GAP}.
\begin{remark}
	It is worth mentioning here that our product in $\cC_n$ is written from left to right, consistent with the order of both the permutations in GAP4 and the quantum circuits. In contrast, this order is the opposite of that used in matrix multiplication within ${\rm Cliff}_n$.
\end{remark}

The key idea is to understand the structure of the centralizers of the $z_1$ and $s_1$ gates in $\cC_n$, namely,
\begin{align*}
	\cent_{\cC_n}(u) =\{g\in\cC_n: ug=gu\}\quad\text{for }u\in\{s_1^2,s_1\}.
\end{align*}
We will prove the following result:
\begin{theorem*}
	For $n\geq 2$, the centralizers of $s_1^2$ and $s_1$ in $\cC_n$ are given by
	\begin{align*}
		\cent_{\cC_n}(u)=
		\geng{g, h_2, \dots, h_n}{s_1, \dots, s_n}{\Lambda_1,\dots, \Lambda_{n-1}},\quad\text{where } g=\begin{cases}
			h_1s_1h_1s_1^3h_1,&\text{if }u=s_1^2,\\
			1,&\text{if }u=s_1.	
		\end{cases}
	\end{align*}
	Moreover $\cent_{\cC_1}(s_1^2)=\langle h_1s_1h_1s_1^3h_1, s_1\rangle$ and $\cent_{\cC_1}(s_1)=\langle s_1\rangle$.
\end{theorem*}

The above theorem implies that $\cent_{\mathcal{C}_n}(s_1^2)$ is isomorphic to the inertia subgroup $\mathcal{I}\mathcal{N}_n$ of $\mathcal{C}_n$, which will be introduced in Section \ref{sec:presentofINn}.
It has been shown in \cite[Theorem IV.4]{Ma} that the set of irreducible representations of $\mathcal{C}_n$ is determined by the irreducible representations of $\mathcal{I}\mathcal{N}_n / \mathcal{P}_n$ and ${\rm Sp}(2n, 2)$, the symplectic group of degree $2n$ over $\mathbb{Z}_2$.
We provide a presentation of $\mathcal{I}\mathcal{N}_n$, which may be useful for future studies.

The remainder of the paper is organized as follows: 
In Section \ref{sec:nf}, we recall the normal form of Clifford operators, which provides a structured way to describe the elements of $\cC_n$.
In Section \ref{sec:rw}, we recall how the rewriting rules operate. 
In Section~\ref{sec:C1C2}, we collect several technical identities needed in the proofs of the main results.
In Section \ref{sec:Css}, we determine the centralizer of the $z_1$ gate in $\cC_n$.
In Section \ref{sec:Cs}, we determine the centralizer of the $s_1$ gate in $\cC_n$.
In Section \ref{sec:perm}, we prove that the conjugacy action of $\cC_n$ on the conjugacy class of $s_1$ in $\cC_n$ induced a faithful homomorphism from $\cC_n$ into ${\rm Sym}_{2(4^n-1)}$.
In Section \ref{sec:presentofINn}, we provide a presentation of $\cI\cN_n$.
In Section \ref{sec:openprob}, we summarize the results of this paper and present
further results related to this work.

\section{Normal form for Clifford operators}\label{sec:nf}

In this section, we use quantum circuits to represent the elements of $\cC_n$.
Following the notation in \cite{Sel}, the circuits are written from left to right. The generators $h_i$, $s_i$, and $\Lambda_i$ are represented by placing the corresponding quantum gates
\scalebox{0.55}{
	\begin{quantikz}
		&\gate{H}& \qw
	\end{quantikz},
	\begin{quantikz}
		&\gate{S}& \qw
	\end{quantikz},
	\begin{quantikz}
		&\ctrl{1}&\qw\\
		&\control{}&\qw
	\end{quantikz}
}     
in appropriate positions within the circuit.

For example, $h_1h_2s_1\Lambda_1$ and $\Lambda_1\Lambda_2$ are represented as
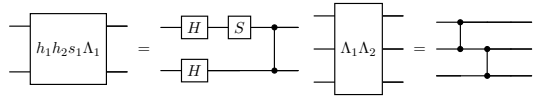
\begin{equation*}
	\scalebox{0.55}{
		\begin{quantikz}
			&\gate[2]{h_1h_2s_1\Lambda_1}&\qw\\
			&							&\qw
		\end{quantikz}
		=	
		\begin{quantikz}
			&\gate{H}&\gate{S}&\ctrl{1}   &\qw\\
			&\gate{H}&\qw     &\control{}&\qw
		\end{quantikz}
		\begin{quantikz}
			&\gate[3]{\Lambda_1\Lambda_2}&\qw\\
			&							 &\qw\\
			&                            &\qw
		\end{quantikz}
		=	
		\begin{quantikz}
			&\ctrl{1}  &\qw       &\qw&\qw\\
			&\control{}&\ctrl{1}  &\qw&\qw\\
			&\qw       &\control{}&\qw&\qw
	\end{quantikz}}
\end{equation*}
and their product $h_1h_2s_1\Lambda_1\Lambda_1\Lambda_2$ simply represented by the concatenation of two circuits:
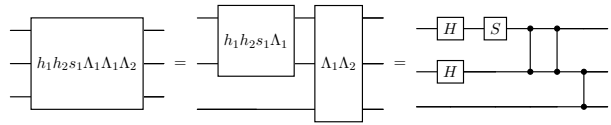
\begin{equation*}
	\scalebox{0.55}{
		\begin{quantikz}
			&\gate[3]{h_1h_2s_1\Lambda_1\Lambda_1\Lambda_2}&\qw\\
			&											   &\qw\\
			&                                 			   &\qw
		\end{quantikz}
		=
		\begin{quantikz}
			&\gate[2]{h_1h_2s_1\Lambda_1}&\gate[3]{\Lambda_1\Lambda_2}&\qw\\
			&							 &    						  &\qw\\
			&\qw						 &							  &\qw
		\end{quantikz}
		=
		\begin{quantikz}
			&\gate{H}&\gate{S}&\ctrl{1}  &\ctrl{1}  &\qw       &\qw\\
			&\gate{H}&\qw     &\control{}&\control{}&\ctrl{1}  &\qw\\
			&\qw     &\qw     &\qw       &\qw		&\control{}&\qw     
		\end{quantikz}
	}.
\end{equation*}
In order to define the normal form, we define below several basic quantum gates and categorize them into five types, labeled from $A$ to $E$:
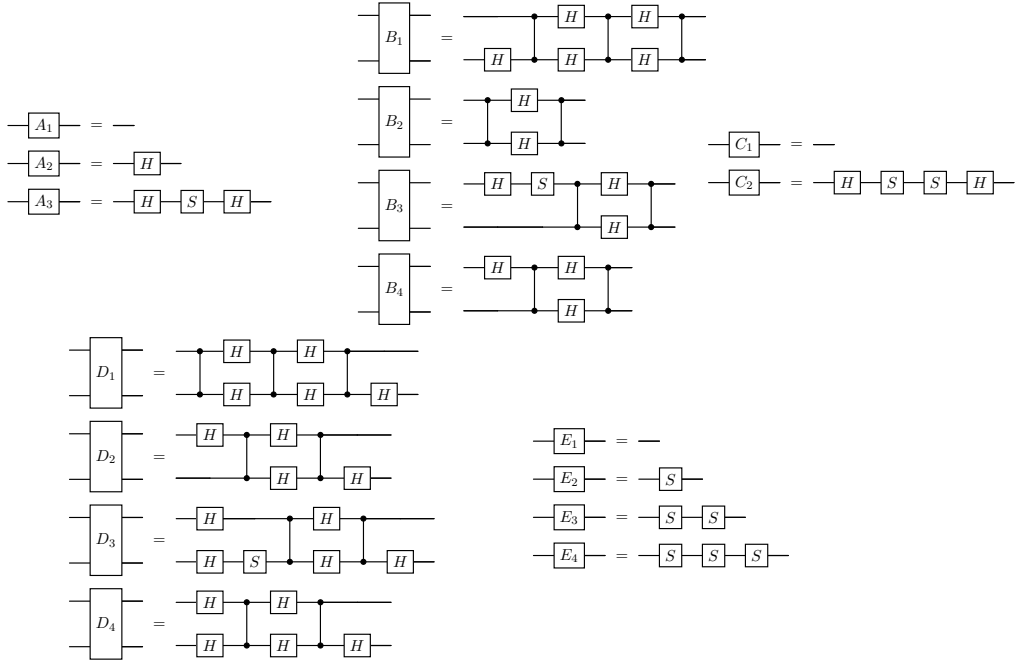
\begin{center}
	\begin{minipage}{0.3\textwidth}
		\scalebox{0.55}{
			\begin{quantikz}
				&\gate{A_1}& \qw        
			\end{quantikz}
			=
			\begin{quantikz}
				&\qw
		\end{quantikz}} 
		
		\scalebox{0.55}{
			\begin{quantikz}
				&\gate{A_2}&\qw    
			\end{quantikz}
			=
			\begin{quantikz}
				&\gate{H}&\qw
		\end{quantikz}} 
		\scalebox{0.55}{
			\begin{quantikz}
				&\gate{A_3}&\qw
			\end{quantikz}
			=
			\begin{quantikz}
				&\gate{H}&\gate{S}&\gate{H}&\qw
		\end{quantikz}} 
	\end{minipage}
	\begin{minipage}{0.3\textwidth}
		\scalebox{0.55}{
			\begin{quantikz}
				&\gate[2]{B_1}&\qw\\
				&             &\qw
			\end{quantikz}
			=
			\begin{quantikz}
				&\qw     &\ctrl{1}  &\gate{H} &\ctrl{1}  &\gate{H}&\ctrl{1}  &\qw\\
				&\gate{H}&\control{}&\gate{H} &\control{}&\gate{H}&\control{}&\qw
		\end{quantikz}} 
		\scalebox{0.55}{
			\begin{quantikz}
				&\gate[2]{B_2}&\\
				&			  &
			\end{quantikz}
			=
			\begin{quantikz}
				&\ctrl{1}  &\gate{H}&\ctrl{1}  &\qw\\
				&\control{}&\gate{H}&\control{}&\qw
		\end{quantikz}} 
		\scalebox{0.55}{
			\begin{quantikz}
				&\gate[2]{B_3}&\\
				&			  &
			\end{quantikz}
			=
			\begin{quantikz}
				&\gate{H}&\gate{S}&\ctrl{1}  &\gate{H}&\ctrl{1}  &\qw\\
				&\qw     &\qw     &\control{}&\gate{H}&\control{}&\qw
		\end{quantikz}}
		\scalebox{0.55}{
			\begin{quantikz}
				&\gate[2]{B_4}&\qw\\
				&             &\qw
			\end{quantikz}
			=
			\begin{quantikz}
				&\gate{H}&\ctrl{1}  &\gate{H}&\ctrl{1}  &\qw\\
				&\qw     &\control{}&\gate{H}&\control{}&\qw
		\end{quantikz}} 
	\end{minipage}
	\begin{minipage}{0.3\textwidth}
		\scalebox{0.55}{
			\begin{quantikz}
				&\gate{C_1}&\qw        
			\end{quantikz}
			=
			\begin{quantikz}
				&\qw
		\end{quantikz}} 
		
		\scalebox{0.55}{
			\begin{quantikz}
				&\gate{C_2}&\qw      
			\end{quantikz}
			=
			\begin{quantikz}
				&\gate{H}&\gate{S}&\gate{S}&\gate{H}&\qw
		\end{quantikz}} 
	\end{minipage}
	
	\begin{minipage}{0.4\textwidth}
		\scalebox{0.55}{
			\begin{quantikz}
				&\gate[2]{D_1}&\qw\\
				&             &\qw
			\end{quantikz}
			=
			\begin{quantikz}
				&\ctrl{1}  &\gate{H}&\ctrl{1}  &\gate{H}&\ctrl{1}  &\qw     &\qw\\
				&\control{}&\gate{H}&\control{}&\gate{H}&\control{}&\gate{H}&\qw
		\end{quantikz}} 
		\scalebox{0.55}{
			\begin{quantikz}
				&\gate[2]{D_2}&\qw\\
				&             &\qw
			\end{quantikz}
			=
			\begin{quantikz}
				&\gate{H}&\ctrl{1}  &\gate{H}&\ctrl{1}  &\qw     &\qw\\
				&\qw     &\control{}&\gate{H}&\control{}&\gate{H}&\qw
		\end{quantikz}} 
		\scalebox{0.55}{
			\begin{quantikz}
				&\gate[2]{D_3}&\qw\\
				&             &\qw
			\end{quantikz}
			=
			\begin{quantikz}
				&\gate{H}&\qw     &\ctrl{1}  &\gate{H}&\ctrl{1}  &\qw     &\qw\\
				&\gate{H}&\gate{S}&\control{}&\gate{H}&\control{}&\gate{H}&\qw
		\end{quantikz}}
		\scalebox{0.55}{
			\begin{quantikz}
				&\gate[2]{D_4}&\qw\\
				&             &\qw
			\end{quantikz}
			=
			\begin{quantikz}
				&\gate{H}&\ctrl{1}  &\gate{H}&\ctrl{1}  &\qw     &\qw\\
				&\gate{H}&\control{}&\gate{H}&\control{}&\gate{H}&\qw
		\end{quantikz}} 
	\end{minipage}
	\begin{minipage}{0.4\textwidth}
		\scalebox{0.55}{
			\begin{quantikz}
				&\gate{E_1}&\qw       
			\end{quantikz}
			=
			\begin{quantikz}
				&\qw 
		\end{quantikz}} 
		
		\scalebox{0.55}{
			\begin{quantikz}
				&\gate{E_2}&\qw       
			\end{quantikz}
			=
			\begin{quantikz}
				&\gate{S}&\qw 
		\end{quantikz}} 
		
		\scalebox{0.55}{
			\begin{quantikz}
				&\gate{E_3}&\qw
			\end{quantikz}
			=
			\begin{quantikz}
				&\gate{S}& \gate{S}&\qw
		\end{quantikz}} 
		\scalebox{0.55}{
			\begin{quantikz}
				&\gate{E_4}&\qw
			\end{quantikz}
			=
			\begin{quantikz}
				&\gate{S}&\gate{S}&\gate{S}&\qw
		\end{quantikz}} 
	\end{minipage}
\end{center}

\begin{definition}
	We say that an $n$-qubit circuit is \emph{($\ell$-layer) $ZX$-normal} if it is of the form:
	\begin{equation*}
		\scalebox{0.55}
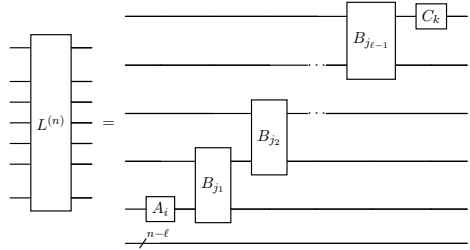
{
			\begin{quantikz}
				&\gate[7]{\up N n}&\qw\\
				&				  &\qw\\
				&				  &\qw\\
				&				  &\qw\\
				&				  &\qw\\
				&				  &\qw\\
				&				  &\qw
			\end{quantikz}
			=
			\begin{quantikz}
				&\qw              &\qw              &\qw              &\qw   &\gate[2]{B_{j_{\ell-1}}}&\gate{C_k}&\gate[2]{D_{l_1}}&\qw              &\qw              &\qw              &\qw   &\qw                  &\qw  	    &\qw\\
				&\qw              &\qw              &\qw              &\push{\cdots}&\qw                     &\qw       &\qw              &\gate[2]{D_{l_2}}&\qw              &\qw              &\qw   &\qw                  &\qw   	&\qw\\
				&\qw              &\qw              &\gate[2]{B_{j_2}}&\push{\cdots}&\qw                     &\qw       &\qw              &\qw              &\gate[2]{D_{l_3}}&\qw              &\qw   &\qw                  &\qw       &\qw\\
				&\qw              &\gate[2]{B_{j_1}}&\qw              &\qw   &\qw                     &\qw       &\qw              &\qw              &\qw              &\gate[2]{D_{l_4}}&\qw   &\qw                  &\qw       &\qw\\
				&\gate{A_i}       &\qw              &\qw              &\qw   &\qw                     &\qw       &\qw              &\qw              &\qw              &\qw              &\push{\cdots}&\qw                  &\qw		&\qw\\
				&\qw              &\qw              &\qw              &\qw   &\qw                     &\qw       &\qw              &\qw              &\qw              &\qw              &\push{\cdots}&\gate[2]{D_{l_{n-1}}}&\qw		&\qw\\
				&\qw              &\qw              &\qw              &\qw   &\qw                     &\qw       &\qw              &\qw              &\qw              &\qw              &\qw   &\qw                  &\gate{E_h}&\qw
		\end{quantikz}} 
	\end{equation*}
	for $1\leq \ell\leq n$ where $i\in\{1,2,3\}, j_1,\dots, j_{\ell-1} \in \{1,2,3,4\}, k\in \{1,2\}, l_1,\dots,l_{n-1}\in\{1,2,3,4\}, h\in \{1,2,3,4\}$.	
	The number $\ell$ indicates the wire containing the gate $A_i$, counted from the top.
\end{definition}

\begin{theorem}\cite[Section 5]{Sel} \label{thm:normform_sel}
	Every $g \in \cC_n$ can be uniquely expressed as \begin{align*} g = \up Nn_g \up N{n-1}_g \cdots \up N1_g, \end{align*} where each $\up N{i}_g$ is a $ZX$-normal form in $\cC_i$. In other words,
	\begin{equation}\label{eq:nf}
		\scalebox{0.55}{		\begin{quantikz}
				&\gate[7]{g}&\qw\\
				&				  &\qw\\
				&				  &\qw\\
				&				  &\qw\\
				&				  &\qw\\
				&				  &\qw\\
				&				  &\qw
			\end{quantikz}=
			\begin{quantikz}
				&\gate[7]{\nf(g)}&\qw\\
				&				  &\qw\\
				&				  &\qw\\
				&				  &\qw\\
				&				  &\qw\\
				&				  &\qw\\
				&				  &\qw
			\end{quantikz}:=
			\begin{quantikz}
				&\gate[4]{\up N{n}_g}&\gate[3]{\up N{n-1}_g}&\push{\cdots}&\gate[2]{\up N{2}_g} &\gate{\up N1_g}     &\qw\\
				&\qw               &\qw                 &\push{\cdots}&\qw                &\qw               &\qw\\
				&\qw               &\qw                 &\qw   &\qw                &\qw               &\qw\\
				&\qw               &\qw                 &\qw   &\qw                &\qw               &\qw
		\end{quantikz}}.
	\end{equation}
\end{theorem}
The order of $\cC_n$ is well-known, and it can also be obtained directly
from Theorem \ref{thm:normform_sel}. For each \(i\), the number of $ZX$-normal forms in
$\cC_i$ is
\begin{align*}
\sum_{\ell=1}^{i}
3\cdot 4^{\ell-1}\cdot 2\cdot 4^{i-1}\cdot 4
=
2^{2i+1}(4^i-1).
\end{align*}
Since the decomposition in Theorem \ref{thm:normform_sel}. is unique, it follows that
\begin{align}\label{eq:orderCn}
|C_n|
=
\prod_{i=1}^{n}2^{2i+1}(4^i-1)
=
2^{n^2+2n}\prod_{i=1}^n (4^i-1).
\end{align}

\section{Rewriting rules}\label{sec:rw}

\begin{definition}
	Consider an $n$-qubit Clifford circuit in normal form:
	\begin{equation*}
		\scalebox{0.45}{
			\begin{quantikz}
				&\push{1}&\qw&\qw&\qw&\qw&\qw&\qw&\qw&\qw&\gate[2]{B_{j_m-1}}&\push{2}&\gate{C_k}&\push{3}&\gate[2]{D_{l_1}}&\qw&\qw&\qw&\qw&\qw&\qw&\qw&\qw&\qw&\qw&\qw&\qw&\push{1}&\gate[6]{\up N{n-1}}&\push{\cdots}&\qw\\
				&\push{1}&\qw&\qw&\qw&\qw&\qw&\qw&\push{\cdots}&\push{2}&&\qw&\qw&\push{1}&&\push{4}&\gate[2]{D_{l_2}}&\qw&\qw&\qw&\qw&\qw&\qw&\qw&\qw&\qw&\qw&\push{1}&&\push{\cdots}&\qw\\
				&\push{1}&\qw&\qw&\qw&\qw&\gate[2]{B_{j_2}}&\push{2}&\push{\cdots}&\qw&\qw&\qw&\qw&\push{1}&\qw&\qw&&\push{4}&\gate[2]{D_{l_3}}&\qw&\qw&\qw&\qw&\qw&\qw&\qw&\qw&\push{1}&&\push{\cdots}&\qw\\
				&\push{1}&\qw&\qw&\gate[2]{B_{j_1}}&\push{2}&&\qw&\qw&\qw&\qw&\qw&\qw&\push{1}&\qw&\qw&\qw&\qw&&\push{4}&\gate[2]{D_{l_4}}&\qw&\qw&\qw&\qw&\qw&\qw&\push{1}&&\push{\cdots}&\qw\\
				&\push{1}&\gate{A_i}&\push{2}&&\qw&\qw&\qw&\qw&\qw&\qw&\qw&\qw&\push{1}&\qw&\qw&\qw&\qw&\qw&\qw&&\push{4}&\push{\cdots}&\qw&\qw&\qw&\qw&\push{1}&&\push{\cdots}&\qw\\
				&\push{1}&\qw&\qw&\qw&\qw&\qw&\qw&\qw&\qw&\qw&\qw&\qw&\push{1}&\qw&\qw&\qw&\qw&\qw&\qw&\qw&\qw&\push{\cdots}&\push{4}&\gate[2]{D_{l_n-1}}&\qw&\qw&\push{1}&&\qw&\qw\\
				&\push{1}&\qw&\qw&\qw&\qw&\qw&\qw&\qw&\qw&\qw&\qw&\qw&\push{1}&\qw&\qw&\qw&\qw&\qw&\qw&\qw&\qw&\qw&\qw&&\push{4}&\gate{E_h}&\qw&\qw&\qw&\qw\\
			\end{quantikz}
		}
	\end{equation*}
	
	We say that a circuit is in \emph{dirty normal form} if it is of the form , except that the circuit may contain some additional gates, subject to the following rules:
	\begin{enumerate}
		\item[$\bullet$] $H$-gates may be added to any wire labelled $1$;
		\item[$\bullet$] $S$-gates may be added to any wire labelled $1$, $2$, $3$ or $4$;
		\item[$\bullet$] $X$-gates may be added to any wire labelled $2$;
		\item[$\bullet$] $CZ$-gates may be added to any pair of adjacent wires, provided that the top wire is labelled $1$, $2$ or $3$, and the bottom wire is labelled $1$.
	\end{enumerate}
\end{definition}

Any dirty normal form can be transformed into an equivalent normal form by applying a finite sequence of rewriting rules in the left-to-right direction. For details on the rewriting rules, refer to \cite[Figure 3--7]{Sel}.
We record below the rules that will be needed in the subsequent sections.
\begin{lemma}\label{lem:ace_sel}
\cite[Figure 3, 4, 6]{Sel}
The following rewriting rules hold:
\begin{center}

\begin{minipage}{0.65\textwidth}
\centering

\makebox[\linewidth][c]{%
\scalebox{0.5}{
\begin{quantikz}
    & \gate{S} & \gate{A_1} & \qw
\end{quantikz}
\;=\;
\begin{quantikz}
    & \gate{A_1} & \gate{S} & \qw
\end{quantikz}}}

\makebox[\linewidth][c]{%
\scalebox{0.5}{
\begin{quantikz}
    & \gate{S} & \gate{A_2} & \qw
\end{quantikz}
\;=\;
\begin{quantikz}
    & \gate{A_3} & \gate{X} & \gate{S}
    & \gate{S} & \gate{S} & \qw
\end{quantikz}}}

\makebox[\linewidth][c]{%
\scalebox{0.5}{
\begin{quantikz}
    & \gate{S} & \gate{A_3} & \qw
\end{quantikz}
\;=\;
\begin{quantikz}
    & \gate{A_2} & \gate{S}
    & \gate{S} & \gate{S} & \qw
\end{quantikz}}}

\end{minipage}

\vspace{1em}

\begin{minipage}{0.76\textwidth}
\centering

\begin{minipage}[t]{0.55\textwidth}
\centering

\scalebox{0.5}{
\begin{quantikz}
    & \gate{X} & \gate{C_1} & \qw
\end{quantikz}
=
\begin{quantikz}
    & \gate{C_2} & \qw
\end{quantikz}}

\scalebox{0.5}{
\begin{quantikz}
    & \gate{X} & \gate{C_2} & \qw
\end{quantikz}
=
\begin{quantikz}
    & \gate{C_1} & \qw
\end{quantikz}}

\scalebox{0.5}{
\begin{quantikz}
    & \gate{S} & \gate{C_1} & \qw
\end{quantikz}
=
\begin{quantikz}
    & \gate{C_1} & \gate{S} & \qw
\end{quantikz}}

\scalebox{0.5}{
\begin{quantikz}
    & \gate{S} & \gate{C_2} & \qw
\end{quantikz}
=
\begin{quantikz}
    & \gate{C_2} & \gate{S}
    & \gate{S} & \gate{S} & \qw
\end{quantikz}}

\end{minipage}
\hfill
\begin{minipage}[t]{0.35\textwidth}
\centering

\scalebox{0.5}{
\begin{quantikz}
    & \gate{S} & \gate{E_1} & \qw
\end{quantikz}
=
\begin{quantikz}
    & \gate{E_2} & \qw
\end{quantikz}}

\scalebox{0.5}{
\begin{quantikz}
    & \gate{S} & \gate{E_2} & \qw
\end{quantikz}
=
\begin{quantikz}
    & \gate{E_3} & \qw
\end{quantikz}}

\scalebox{0.5}{
\begin{quantikz}
    & \gate{S} & \gate{E_3} & \qw
\end{quantikz}
=
\begin{quantikz}
    & \gate{E_4} & \qw
\end{quantikz}}

\scalebox{0.5}{
\begin{quantikz}
    & \gate{S} & \gate{E_4} & \qw
\end{quantikz}
=
\begin{quantikz}
    & \gate{E_1} & \qw
\end{quantikz}}

\end{minipage}

\end{minipage}

\end{center}
\end{lemma}

	\begin{lemma}\label{lem:sb_sel} \cite[Figure 4]{Sel} The following rewriting rules hold:
		\begin{center}
			\begin{minipage}{0.4\textwidth}
				\scalebox{0.55}{
					\begin{quantikz}
						& \gate{S} & \gate[2]{B_1}  &\qw  \\
						&\qw       &             &\qw 
					\end{quantikz}
					=
					\begin{quantikz}
						& \gate[2]{B_1} &\qw        &\qw        &\qw        &\qw  \\
						&            & \gate{H}  & \gate{S}  &\gate{H}   &\qw  
				\end{quantikz}}
				\scalebox{0.55}{
					\begin{quantikz}
						& \gate{S} & \gate[2]{B_2}  &\qw  \\
						&\qw       &             &\qw 
					\end{quantikz}
					=
					\begin{quantikz}
						& \gate[2]{B_3} & \gate{X}       &\qw        &\qw        &\qw    &\qw  &\qw\\
						&            & \gate{S}  & \gate{S}  &\gate{S}   &  \gate{H}&\gate{S}&\qw
				\end{quantikz}}
				\scalebox{0.5}{
					\begin{quantikz}
						& \gate{S} & \gate[2]{B_3}  &\qw  \\
						&\qw       &             &\qw 
					\end{quantikz}
					=
					\begin{quantikz}
						& \gate[2]{B_2} &\qw        &\qw        &\qw        &\qw  \\
						&            & \gate{S}  & \gate{H}  &\gate{S}   &\qw  
				\end{quantikz}}
				\scalebox{0.5}{
					\begin{quantikz}
						& \gate{S} & \gate[2]{B_4}  &\qw  \\
						&\qw       &             &\qw 
					\end{quantikz}
					=
					\begin{quantikz}
						& \gate[2]{B_4} &\qw        &\qw        &\qw        &\qw  \\
						&            & \gate{H}  & \gate{S}  &\gate{H}   &\qw  
				\end{quantikz}}
			\end{minipage}
\begin{minipage}{0.4\textwidth}
	\scalebox{0.5}{
		\begin{quantikz}
			& \gate{H} & \gate[2]{B_1} & \qw \\
			& \qw      &                & \qw
		\end{quantikz}
		=
		\begin{quantikz}
			& \gate[2]{B_1} & \qw      & \qw \\
			&                & \gate{H} & \qw
		\end{quantikz}
	}

	\scalebox{0.5}{
		\begin{quantikz}
			& \gate{H} & \gate[2]{B_2} & \qw \\
			& \qw      &                & \qw
		\end{quantikz}
		=
		\begin{quantikz}
			& \gate[2]{B_4} & \qw \\
			&                & \qw
		\end{quantikz}
	}

	\scalebox{0.5}{
		\begin{quantikz}
			& \gate{H} & \gate[2]{B_3} & \qw \\
			& \qw      &                & \qw
		\end{quantikz}
		=
		\begin{quantikz}
			& \gate[2]{B_3}
			& \gate{X}
			& \qw
			& \qw
			& \qw
			& \qw
			& \qw \\
			&
			& \gate{S}
			& \gate{S}
			& \gate{S}
			& \gate{H}
			& \gate{S}
			& \qw
		\end{quantikz}
	}

	\scalebox{0.5}{
		\begin{quantikz}
			& \gate{H} & \gate[2]{B_4} & \qw \\
			& \qw      &                & \qw
		\end{quantikz}
		=
		\begin{quantikz}
			& \gate[2]{B_2} & \qw \\
			&                & \qw
		\end{quantikz}
	}
\end{minipage}		\end{center}
	\end{lemma}
	
	\begin{lemma}\label{lem:sd_sel}\cite[Figure 6]{Sel} The following rewriting rules hold:
		\begin{center}
			\begin{minipage}{0.4\textwidth}
				\scalebox{0.5}{
					\begin{quantikz}
						&\qw          &   \gate[2]{D_1}  &\qw  \\
						& \gate{S} &            &\qw 
					\end{quantikz}
					=
					\begin{quantikz}
						& \gate[2]{D_1} & \gate{H}       &\gate{S}&\gate{H}&\qw          \\
						&            &\qw   &\qw   &\qw&\qw
				\end{quantikz}}
				\scalebox{0.5}{
					\begin{quantikz}
						&\qw          &   \gate[2]{D_2}  &\qw  \\
						& \gate{S} &            &\qw 
					\end{quantikz}
					=
					\begin{quantikz}
						& \gate[2]{D_3} & \gate{S}       &\gate{S}&\gate{S}& \gate{H} &\gate{S}&\qw         \\
						&            &  \gate{S} & \gate{S}  &\qw&\qw&\qw&\qw
				\end{quantikz}}
				\scalebox{0.5}{
					\begin{quantikz}
						&\qw          &   \gate[2]{D_3}  &\qw  \\
						& \gate{S} &            &\qw 
					\end{quantikz}
					=
					\begin{quantikz}
						& \gate[2]{D_2} & \gate{S}       &\gate{H}&\gate{S}&\qw          \\
						&            &\qw   &\qw   &\qw&\qw
				\end{quantikz}}
				\scalebox{0.5}{
					\begin{quantikz}
						&\qw          &   \gate[2]{D_4}  &\qw  \\
						& \gate{S} &            &\qw 
					\end{quantikz}
					=
					\begin{quantikz}
						& \gate[2]{D_4} & \gate{H}       &\gate{S}&\gate{H}&\qw          \\
						&            &\qw   &\qw   &\qw&\qw
				\end{quantikz}}
			\end{minipage}
\begin{minipage}{0.4\textwidth}
	\scalebox{0.5}{
		\begin{quantikz}
			& \qw      & \gate[2]{D_1} & \qw \\
			& \gate{H} &                & \qw
		\end{quantikz}
		=
		\begin{quantikz}
			& \gate[2]{D_1} & \gate{H} & \qw \\
			&                & \qw      & \qw
		\end{quantikz}
	}

	\scalebox{0.5}{
		\begin{quantikz}
			& \qw      & \gate[2]{D_2} & \qw \\
			& \gate{H} &                & \qw
		\end{quantikz}
		=
		\begin{quantikz}
			& \gate[2]{D_4} & \qw \\
			&                & \qw
		\end{quantikz}
	}

	\scalebox{0.5}{
		\begin{quantikz}
			& \qw      & \gate[2]{D_3} & \qw \\
			& \gate{H} &                & \qw
		\end{quantikz}
		=
		\begin{quantikz}
			& \gate[2]{D_3}
			& \gate{S}
			& \gate{S}
			& \gate{S}
			& \gate{H}
			& \gate{S}
			& \qw \\
			&
			& \gate{S}
			& \gate{S}
			& \qw
			& \qw
			& \qw
			& \qw
		\end{quantikz}
	}

	\scalebox{0.5}{
		\begin{quantikz}
			& \qw      & \gate[2]{D_4} & \qw \\
			& \gate{H} &                & \qw
		\end{quantikz}
		=
		\begin{quantikz}
			& \gate[2]{D_2} & \qw \\
			&                & \qw
		\end{quantikz}
	}
\end{minipage}
		\end{center}
	\end{lemma}

\begin{lemma} \cite[Figure 6]{Sel}\label{lem:sd_propagation}
For every $j\in\{1,2,3,4\}$, the following rewriting rule holds:
\begin{center}
	\scalebox{0.55}{
		\begin{quantikz}
			& \gate{S} & \gate[2]{D_j} & \qw \\
			& \qw      &               & \qw
		\end{quantikz}
		=
		\begin{quantikz}
			& \gate[2]{D_j} & \qw      & \qw \\
			&               & \gate{S} & \qw
		\end{quantikz}
	}
\end{center}
\end{lemma}

The process of transforming $\nf(f)\nf(g)$ into $\nf(fg)$ serves as a proof of the equality
\begin{align*}
	\nf(f)\nf(g)=\nf(fg).
\end{align*}
Selinger discovered a set of relations \cite[Figure 8]{Sel} that can derive all the rewriting rules, thus offering a presentation of $\cC_n$, which is provided in Appendix \ref{sec:appendix}.

The normal form of identity $1\in\cC_n$ is given by
\begin{equation*}
	\scalebox{0.55}{
		\begin{quantikz}
			&\gate[5]{\nf(1)}&\qw\\
			&&\qw\\
			&&\qw\\
			&&\qw\\
			&&\qw
		\end{quantikz}= 
		\begin{quantikz}
			&\gate[5]{\up In}&\gate[4]{\up I{n-1}}&\push{\cdots}&\gate[2]{\up I2}&\gate{\up I1}&\qw\\
			&                &                    &\push{\cdots}&&\qw          &\qw\\
			&                &                    &\push{\cdots}&\qw&\qw          &\qw\\
			&                &                    &\qw          &\qw&\qw          &\qw\\
			&                &\qw                 &\qw          &\qw&\qw          &\qw 
		\end{quantikz}
		
	}
\end{equation*}
where $\up I1 = A_1 C_1 E_1$ and for each $i \in \{n, \dots, 2\}$, $\up Ii \in \cC_i$ is
\begin{equation*}
	\scalebox{0.55}{	
		\begin{quantikz}
			&\gate[5]{\up Ii}&\qw\\
			&&\qw\\
			&&\qw\\
			&&\qw\\
			&&\qw
		\end{quantikz}= 
		\begin{quantikz}
			&\qw&\qw&\qw&\qw&\gate[2]{B_1}&\gate{C_1}&\gate[2]{D_1}&\qw&\qw&\qw&\qw&\qw\\
			&\qw&\qw&\qw&\push{\cdots}&\qw&\qw&\qw&\push{\cdots}&\qw&\qw&\qw&\qw\\
			&\qw&\qw&\gate[2]{B_1}&\push{\cdots}&\qw&\qw&\qw&\push{\cdots}&\gate[2]{D_1}&\qw&\qw&\qw\\
			&\qw&\gate[2]{B_1}&&\qw&\qw&\qw&\qw&\qw&&\gate[2]{D_1}&\qw&\qw\\
			&\gate{A_1}&&\qw&\qw&\qw&\qw&\qw&\qw&\qw&&\gate{E_1}&
		\end{quantikz}	
	}.
\end{equation*}
Since multiplying the generators of $\cC_n$ with a normal form produces a dirty normal form, the process of converting any word $g = g_1 \cdots g_m \in \cC_n$ into its normal form, $\nf(g)$, proceeds as follows:
\begin{align*}
	g&=g_1\cdots g_m\cdot \nf(1)=g_1\cdots g_{m-1}\cdot \nf(g_m)=g_1\cdots g_{m-2} \cdot \nf(g_{m-1}g_m)\\
	&=\cdots=\nf(g_1\cdots g_m).
\end{align*}

We next prove the following lemma, which allows us to regard
$\cC_m$ as a subgroup of $\cC_n$ whenever $m\leq n$.
Although this inclusion is natural, it is not immediate from the
presentation of $\cC_n$, since one must verify that the relations
involving the remaining qubits impose no additional relations on
the generators of $\cC_m$.

\begin{lemma} The $\cC_{n-1}$ is isomorphic to the following subgroup of $\cC_n$:
	\begin{align*} \geng{h_1, \dots, h_{n-1}}{s_1, \dots, s_{n-1}}{\Lambda_1, \dots, \Lambda_{n-2}} . \end{align*}
\end{lemma}

\begin{proof} Let $\cG_{n-1}$ be the subgroup as described in the statement. It is easy to see that the map $\cC_{n-1} \to \cG_{n-1}$, which sends $h_i \mapsto h_i$, $s_i \mapsto s_i$, and $\Lambda_j \mapsto \Lambda_j$ for all relevant indices, is a surjective homomorphism.
	
	Every element $g \in \cG_{n-1}$ can be expressed in the form \begin{align*} \nf(g) = \up{I}n \nf'(g), \end{align*} where $\nf'(g)$ is the normal form of $g$ within $\cC_{n-1}$. In other words, the orders of $\cG_{n-1}$ and $\cC_{n-1}$ are the same, which completes the proof.
\end{proof}

We conclude this section by recording several identities that will be used
in the subsequent sections. These identities also serve to illustrate the
rewriting process described above. 
All of these identities involve only two qubits and may therefore be
regarded as identities in $\cC_2$. Consequently, they can be verified
either by direct matrix computation in ${\rm Cliff}_2$, up to global phase,
or by using GAP4 to compute in $\cC_2$.

\begin{lemma}  \label{lem:ssb}The following rewriting rules hold: ($U=h_2s_2s_2h_2s_2s_2$) 		\begin{center}
			\begin{minipage}{0.4\textwidth}
				\scalebox{0.55}{
					\begin{quantikz}
						&\gate{S}& \gate{S} & \gate[2]{B_1}  &\qw  \\
						&\qw&\qw       &             &\qw 
					\end{quantikz}
					=
					\begin{quantikz}
						& \gate[2]{B_1} &\qw        &\qw        &\qw        &\qw  &\qw\\
						&            & \gate{H}  & \gate{S}  &\gate{S}   & \gate{H} &\qw
				\end{quantikz}}
				\scalebox{0.55}{
					\begin{quantikz}
						&\gate{S}& \gate{S} & \gate[2]{B_2}  &\qw  \\
						&\qw     &\qw       &                &\qw 
					\end{quantikz}
					=
					\begin{quantikz}
						& \gate[2]{B_2} & \gate{X}       &\qw        &\qw\\
						&               & \gate{S}  & \gate{S}  &\qw
				\end{quantikz}}
				\scalebox{0.55}{
					\begin{quantikz}
						&\gate{S}& \gate{S} & \gate[2]{B_3}  &\qw  \\
						&\qw     &\qw       &             &\qw 
					\end{quantikz}
					=
					\begin{quantikz}
						& \gate[2]{B_3} & \gate{X}       &\qw         \\
						&            & \gate{U}  &\qw 
				\end{quantikz}}
				\scalebox{0.55}{
					\begin{quantikz}
						&\gate{S}& \gate{S} & \gate[2]{B_4}  &\qw  \\
						&\qw     &\qw       &             &\qw 
					\end{quantikz}
					=
					\begin{quantikz}
						& \gate[2]{B_4} &\qw        &\qw        &\qw  &\qw      &\qw  \\
						&            & \gate{H}  & \gate{S} &\gate{S} &\gate{H}   &\qw  
				\end{quantikz}}
			\end{minipage}
			\begin{minipage}{0.4\textwidth}
				\scalebox{0.55}{
					\begin{quantikz}
						&\gate{H}&\gate{S}&\gate{S}& \gate{H} & \gate[2]{B_1}  &\qw  \\
						&\qw     &\qw     &\qw	   &\qw       &             &\qw 
					\end{quantikz}
					=
					\begin{quantikz}
						& \gate[2]{B_1} &\qw        &\qw  &\qw       \\
						&            & \gate{S}  &\gate{S}  &\qw
				\end{quantikz}}
				\scalebox{0.55}{
					\begin{quantikz}
						&\gate{H}&\gate{S}&\gate{S}& \gate{H} & \gate[2]{B_2}  &\qw  \\
						&\qw     &\qw     &\qw     &\qw       &             &\qw 
					\end{quantikz}
					=
					\begin{quantikz}
						& \gate[2]{B_2} &\qw        &\qw        &\qw  &\qw      &\qw  \\
						&            & \gate{H}  & \gate{S}& \gate{S} &\gate{H}   &\qw  
				\end{quantikz}}
				\scalebox{0.55}{
					\begin{quantikz}
						&\gate{H}&\gate{S}&\gate{S}& \gate{H} & \gate[2]{B_3}  &\qw  \\
						&\qw&\qw&\qw&\qw       &\qw             &\qw 
					\end{quantikz}
					=
					\begin{quantikz}
						& \gate[2]{B_3} & \gate{X}       &\qw        &\qw          \\
						&            & \gate{S}  & \gate{S}  &\qw  
				\end{quantikz}}
				\scalebox{0.55}{
					\begin{quantikz}
						&\gate{H}&\gate{S}&\gate{S}& \gate{H} & \gate[2]{B_4}  &\qw  \\
						&\qw&\qw&\qw&\qw       &             &\qw 
					\end{quantikz}
					=
					\begin{quantikz}
						& \gate[2]{B_4} & \gate{X}  &\qw  &\qw  \\
						&            &  \gate{S}&\gate{S}&\qw
				\end{quantikz}}
			\end{minipage}
		\end{center}
	\end{lemma}

\begin{proof}
We label the identities in the left and right columns by $(L1)$--$(L4)$ and $(R1)$--$(R4)$, respectively, from top to bottom.

Proof of $(L1)$: It follows from Lemma \ref{lem:sb_sel} that 
\[
\scalebox{0.55}{$
\vcenter{\hbox{
\begin{quantikz}[row sep=0.28cm, column sep=0.13cm]
    & \gate{S} & \gate{S} & \gate[2]{B_1} & \qw \\
    & \qw      & \qw      &                & \qw
\end{quantikz}
}}
\;=\;
\vcenter{\hbox{
\begin{quantikz}[row sep=0.28cm, column sep=0.13cm]
    & \gate{S} & \gate[2]{B_1} & \qw      & \qw      & \qw      & \qw \\
    & \qw      &                & \gate{H} & \gate{S} & \gate{H} & \qw
\end{quantikz}
}}
\;=\;
\vcenter{\hbox{
\begin{quantikz}[row sep=0.28cm, column sep=0.13cm]
    & \gate[2]{B_1} & \qw      & \qw      & \qw
                      & \qw      & \qw      & \qw&\qw \\
    &                & \gate{H} & \gate{S} & \gate{H}
                      & \gate{H} & \gate{S} & \gate{H}&\qw
\end{quantikz}
}}
\;=\;
\vcenter{\hbox{
\begin{quantikz}[row sep=0.28cm, column sep=0.13cm]
    & \gate[2]{B_1} & \qw      & \qw      & \qw      & \qw &\qw\\
    &                & \gate{H} & \gate{S} & \gate{S} & \gate{H}&\qw
\end{quantikz}
}}
.
$}
\]

Proof of $(L2)$: It follows from Lemma \ref{lem:sb_sel} that 
\[
\scalebox{0.55}{$
\vcenter{\hbox{
\begin{quantikz}[row sep=0.28cm, column sep=0.11cm]
    & \gate{S}
    & \gate{S}
    & \gate[2]{B_2}
    & \gate{X}
    & \qw \\
    & \qw
    & \qw
    &
    & \qw
    & \qw
\end{quantikz}
}}
\;=\;
\vcenter{\hbox{
\begin{quantikz}[row sep=0.28cm, column sep=0.11cm]
    & \gate{S}
    & \gate[2]{B_3}
    & \gate{X}
    & \qw
    & \qw
    & \qw
    & \qw \\
    & \qw
    &
    & \gate{S^3}
    & \gate{H}
    & \gate{S}
    & \qw
    & \qw
\end{quantikz}
}}
\;=\;
\vcenter{\hbox{
\begin{quantikz}[row sep=0.28cm, column sep=0.11cm]
    & \gate[2]{B_2}
    & \gate{X}
    & \qw
    & \qw
    & \qw
    & \qw
    & \qw
    & \qw \\
    &
    & \gate{S}
    & \gate{H}
    & \gate{S}
    & \gate{S^3}
    & \gate{H}
    & \gate{S}
    & \qw
\end{quantikz}
}}
\;=\;
\vcenter{\hbox{
\begin{quantikz}[row sep=0.28cm, column sep=0.11cm]
    & \gate[2]{B_2}
    & \gate{X}
    & \qw
    & \qw \\
    &
    & \gate{S}
    & \gate{S}
    & \qw
\end{quantikz}
}}
.
$}
\]

Proof of $(L3)$: It follows from Lemma \ref{lem:sb_sel} that 

\[
\scalebox{0.55}{$
\vcenter{\hbox{
\begin{quantikz}[row sep=0.28cm, column sep=0.11cm]
    & \gate{S}
    & \gate{S}
    & \gate[2]{B_3}
    & \qw \\
    & \qw
    & \qw
    &
    & \qw
\end{quantikz}
}}
\;=\;
\vcenter{\hbox{
\begin{quantikz}[row sep=0.28cm, column sep=0.11cm]
    & \gate{S}
    & \gate[2]{B_2}
    & \qw
    & \qw
    & \qw
    & \qw \\
    & \qw
    &
    & \gate{S}
    & \gate{H}
    & \gate{S}
    & \qw
\end{quantikz}
}}
\;=\;
\vcenter{\hbox{
\begin{quantikz}[row sep=0.28cm, column sep=0.11cm]
    & \gate[2]{B_3}
    & \gate{X}
    & \qw
    & \qw
    & \qw
    & \qw
    & \qw
    & \qw \\
    &
    & \gate{S^3}
    & \gate{H}
    & \gate{S^2}
    & \gate{H}
    & \gate{S}
    & \qw
    & \qw
\end{quantikz}
}}
.
$}
\]
Finally, note that  $s_2^3h_2s_2^2h_2s_2=h_2s_2^2h_2s_2^2$.

Proof of $(L4)$: It follows from Lemma \ref{lem:sb_sel} that 
\[
\scalebox{0.55}{$
\vcenter{\hbox{
\begin{quantikz}[row sep=0.28cm, column sep=0.11cm]
    & \gate{S}
    & \gate{S}
    & \gate[2]{B_4}
    & \qw \\
    & \qw
    & \qw
    &
    & \qw
\end{quantikz}
}}
\;=\;
\vcenter{\hbox{
\begin{quantikz}[row sep=0.28cm, column sep=0.11cm]
    & \gate{S}
    & \gate[2]{B_4}
    & \qw
    & \qw
    & \qw
    & \qw \\
    & \qw
    &
    & \gate{H}
    & \gate{S}
    & \gate{H}
    & \qw
\end{quantikz}
}}
\;=\;
\vcenter{\hbox{
\begin{quantikz}[row sep=0.28cm, column sep=0.11cm]
    & \gate[2]{B_4}
    & \qw
    & \qw
    & \qw
    & \qw
    & \qw
    & \qw&\qw \\
    &
    & \gate{H}
    & \gate{S}
    & \gate{H}
    & \gate{H}
    & \gate{S}
    & \gate{H}&\qw
\end{quantikz}
}}
\;=\;
\vcenter{\hbox{
\begin{quantikz}[row sep=0.28cm, column sep=0.11cm]
    & \gate[2]{B_4}
    & \qw
    & \qw
    & \qw
    & \qw&\qw \\
    &
    & \gate{H}
    & \gate{S}
    & \gate{S}
    & \gate{H}&\qw
\end{quantikz}
}}
.
$}
\]

Proof of $(R1)$: It follows from Lemma \ref{lem:sb_sel} that 
\[
\scalebox{0.55}{$
\begin{aligned}
&
\vcenter{\hbox{
\begin{quantikz}[row sep=0.28cm, column sep=0.09cm]
    & \gate{H} & \gate{S} & \gate{S} & \gate{H}
    & \gate[2]{B_1} & \qw \\
    & \qw & \qw & \qw & \qw & & \qw
\end{quantikz}
}}
\;=\;
\vcenter{\hbox{
\begin{quantikz}[row sep=0.28cm, column sep=0.09cm]
    & \gate{H} & \gate{S} & \gate{S}
    & \gate[2]{B_1} & \qw & \qw \\
    & \qw & \qw & \qw & & \gate{H}& \qw 
\end{quantikz}
}}
\;=\;
\vcenter{\hbox{
\begin{quantikz}[row sep=0.28cm, column sep=0.09cm]
    & \gate{H} & \gate{S}
    & \gate[2]{B_1}
    & \qw & \qw & \qw & \qw & \qw \\
    & \qw & \qw
    &
    & \gate{H} & \gate{S} & \gate{H} & \gate{H}& \qw 
\end{quantikz}
}}
\\[0.6em]
&
\phantom{
\vcenter{\hbox{
\begin{quantikz}[row sep=0.28cm, column sep=0.09cm]
    & \gate{H} & \gate{S} & \gate{S} & \gate{H}
    & \gate[2]{B_1} & \qw& \qw  \\
    & \qw & \qw & \qw & \qw & & \qw& \qw 
\end{quantikz}
}}
}
\;=\;
\vcenter{\hbox{
\begin{quantikz}[row sep=0.28cm, column sep=0.09cm]
    & \gate{H}
    & \gate[2]{B_1}
    & \qw & \qw & \qw & \qw & \qw & \qw& \qw & \qw  \\
    & \qw
    &
    & \gate{H} & \gate{S} & \gate{H}
    & \gate{H} & \gate{S} & \gate{H}&\gate{H}& \qw 
\end{quantikz}
}}
\;=\;
\vcenter{\hbox{
\begin{quantikz}[row sep=0.28cm, column sep=0.09cm]
    & \gate[2]{B_1}
    & \qw & \qw & \qw & \qw & \qw & \qw & \qw& \qw & \qw  \\
    &
    & \gate{H} & \gate{H} & \gate{S} & \gate{H}
    & \gate{H} & \gate{S} & \gate{H}&\gate{H}& \qw 
\end{quantikz}
}}
\;=\;
\vcenter{\hbox{
\begin{quantikz}[row sep=0.28cm, column sep=0.09cm]
    & \gate[2]{B_1} & \qw & \qw & \qw \\
    & & \gate{S} & \gate{S}& \qw 
\end{quantikz}
}}
.
\end{aligned}
$}
\]

Proof of $(R2)$: It follows from Lemma \ref{lem:sb_sel} that 
\[
\scalebox{0.55}{$
\begin{aligned}
\vcenter{\hbox{
\begin{quantikz}[row sep=0.28cm, column sep=0.09cm]
    & \gate{H} & \gate{S} & \gate{S} & \gate{H}
    & \gate[2]{B_2} & \qw \\
    & \qw & \qw & \qw & \qw & & \qw
\end{quantikz}
}}
&=
\vcenter{\hbox{
\begin{quantikz}[row sep=0.28cm, column sep=0.09cm]
    & \gate{H} & \gate{S} & \gate{S}
    & \gate[2]{B_4} & \qw \\
    & \qw & \qw & \qw & & \qw
\end{quantikz}
}}
\;=\;
\vcenter{\hbox{
\begin{quantikz}[row sep=0.28cm, column sep=0.09cm]
    & \gate{H} & \gate{S}
    & \gate[2]{B_4} & \qw & \qw & \qw & \qw\\
    & \qw & \qw & & \gate{H} & \gate{S} & \gate{H}& \qw
\end{quantikz}
}}
\\[0.6em]
\phantom{
\vcenter{\hbox{
\begin{quantikz}[row sep=0.28cm, column sep=0.09cm]
    & \gate{H} & \gate{S} & \gate{S} & \gate{H}
    & \gate[2]{B_2} & \qw \\
    & \qw & \qw & \qw & \qw & & \qw
\end{quantikz}
}}
}
&=
\vcenter{\hbox{
\begin{quantikz}[row sep=0.28cm, column sep=0.09cm]
    & \gate{H}
    & \gate[2]{B_4}
    & \qw & \qw & \qw & \qw & \qw & \qw& \qw \\
    & \qw & & \gate{H} & \gate{S} & \gate{H}
    & \gate{H} & \gate{S} & \gate{H}& \qw
\end{quantikz}
}}
\;=\;
\vcenter{\hbox{
\begin{quantikz}[row sep=0.28cm, column sep=0.09cm]
    & \gate[2]{B_2}
    & \qw & \qw & \qw & \qw & \qw & \qw & \qw\\
    & & \gate{H} & \gate{S} & \gate{H}
    & \gate{H} & \gate{S} & \gate{H}& \qw
\end{quantikz}
}}
\;=\;
\vcenter{\hbox{
\begin{quantikz}[row sep=0.28cm, column sep=0.09cm]
    & \gate[2]{B_2}
    & \qw & \qw & \qw & \qw & \qw\\
    & & \gate{H} & \gate{S} & \gate{S} & \gate{H}& \qw
\end{quantikz}
}}
.
\end{aligned}
$}
\]

Proof of $(R3)$: It follows from Lemma \ref{lem:sb_sel} that 
\[
\scalebox{0.55}{$
\begin{aligned}
&
\vcenter{\hbox{
\begin{quantikz}[row sep=0.28cm, column sep=0.08cm]
    & \gate{H}
    & \gate{S}
    & \gate{S}
    & \gate{H}
    & \gate[2]{B_3}
    & \qw \\
    & \qw & \qw & \qw & \qw & & \qw
\end{quantikz}
}}
\;=\;
\vcenter{\hbox{
\begin{quantikz}[row sep=0.28cm, column sep=0.08cm]
    & \gate{H}
    & \gate{S}
    & \gate{S}
    & \gate[2]{B_3}
    & \gate{X}
    & \qw & \qw & \qw \\
    & \qw & \qw & \qw
    &
    & \gate{S^3}
    & \gate{H}
    & \gate{S}
    & \qw
\end{quantikz}
}}
\;=\;
\vcenter{\hbox{
\begin{quantikz}[row sep=0.28cm, column sep=0.08cm]
    & \gate{H}
    & \gate{S}
    & \gate[2]{B_2}
    & \gate{X}
    & \qw & \qw & \qw
    & \qw & \qw & \qw \\
    & \qw & \qw
    &
    & \gate{S}
    & \gate{H}
    & \gate{S}
    & \gate{S^3}
    & \gate{H}
    & \gate{S}
    & \qw
\end{quantikz}
}}
\\[0.7em]
&
\qquad =
\vcenter{\hbox{
\begin{quantikz}[row sep=0.28cm, column sep=0.08cm]
    & \gate{H}
    & \gate[2]{B_3}
    & \gate{X}
    & \gate{X}
    & \qw & \qw & \qw
    & \qw & \qw & \qw
    & \qw & \qw  \\
    & \qw
    &
    & \gate{S^3}
    & \gate{H}
    & \gate{S}
    & \gate{S}
    & \gate{H}
    & \gate{S}
    & \gate{S^3}
    & \gate{H}
    & \gate{S}
    & \qw
\end{quantikz}
}}
\;=\;
\vcenter{\hbox{
\begin{quantikz}[row sep=0.28cm, column sep=0.08cm]
    & \gate[2]{B_3}
    & \gate{X}
    & \gate{X}
    & \gate{X}
    & \qw & \qw & \qw
    & \qw & \qw & \qw
    & \qw & \qw & \qw
    & \qw & \qw \\
    &
    & \gate{S^3}
    & \gate{H}
    & \gate{S}
    & \gate{S^3}
    & \gate{H}
    & \gate{S}
    & \gate{S}
    & \gate{H}
    & \gate{S}
    & \gate{S^3}
    & \gate{H}
    & \gate{S}
    & \qw
    & \qw
\end{quantikz}
}}
\;=\;
\vcenter{\hbox{
\begin{quantikz}[row sep=0.28cm, column sep=0.08cm]
    & \gate[2]{B_3}
    & \gate{X}
    & \qw & \qw \\
    &
    & \gate{S}
    & \gate{S}
    & \qw
\end{quantikz}
}}
.
\end{aligned}
$}
\]

Proof of $(R4)$: It follows from Lemma \ref{lem:sb_sel} that 
\[
\scalebox{0.55}{$
\vcenter{\hbox{
\begin{quantikz}[row sep=0.28cm, column sep=0.09cm]
    & \gate{H}
    & \gate{S}
    & \gate{S}
    & \gate{H}
    & \gate[2]{B_4}
    & \qw \\
    & \qw & \qw & \qw & \qw & & \qw
\end{quantikz}
}}
\;=\;
\vcenter{\hbox{
\begin{quantikz}[row sep=0.28cm, column sep=0.09cm]
    & \gate{H}
    & \gate{S}
    & \gate{S}
    & \gate[2]{B_2}
    & \qw \\
    & \qw & \qw & \qw & & \qw
\end{quantikz}
}}
\;=\;
\vcenter{\hbox{
\begin{quantikz}[row sep=0.28cm, column sep=0.09cm]
    & \gate{H}
    & \gate{S}
    & \gate[2]{B_3}
    & \gate{X}
    & \qw
    & \qw
    & \qw \\
    & \qw
    & \qw
    &
    & \gate{S^3}
    & \gate{H}
    & \gate{S}
    & \qw
\end{quantikz}
}}
\;=\;
\vcenter{\hbox{
\begin{quantikz}[row sep=0.28cm, column sep=0.09cm]
    & \gate{H}
    & \gate[2]{B_2}
    & \gate{X}
    & \qw
    & \qw
    & \qw
    & \qw
    & \qw
    & \qw \\
    & \qw
    &
    & \gate{S}
    & \gate{H}
    & \gate{S}
    & \gate{S^3}
    & \gate{H}
    & \gate{S}
    & \qw
\end{quantikz}
}}
\;=\;
\vcenter{\hbox{
\begin{quantikz}[row sep=0.28cm, column sep=0.09cm]
    & \gate[2]{B_4}
    & \gate{X}
    & \qw
    & \qw \\
    &
    & \gate{S}
    & \gate{S}
    & \qw
\end{quantikz}
}}
.
$}
\]
\end{proof}

\begin{lemma} \label{lem:ssd} The following rewriting rules hold: ($U=h_1s_1s_1h_1s_1s_1$)
		\begin{center}
			\begin{minipage}{0.4\textwidth}
				\scalebox{0.55}{
					\begin{quantikz}
						&\qw         &\qw &   \gate[2]{D_1}  &\qw  \\
						& \gate{S}&\gate{S} &            &\qw 
					\end{quantikz}
					=
					\begin{quantikz}
						& \gate[2]{D_1} & \gate{H}    &\gate{S}   &\gate{S}&\gate{H}&\qw          \\
						&            &\qw   &\qw &\qw  &\qw&\qw
				\end{quantikz}}
				\scalebox{0.55}{
					\begin{quantikz}
						&\qw         &\qw &   \gate[2]{D_2}  &\qw  \\
						& \gate{S}&\gate{S} &            &\qw 
					\end{quantikz}
					=
					\begin{quantikz}
						& \gate[2]{D_2} & \gate{S}       &\gate{S}&\qw       \\
						&            &  \gate{S} & \gate{S}  &\qw
				\end{quantikz}}
				\scalebox{0.55}{
					\begin{quantikz}
						&\qw         &\qw &   \gate[2]{D_3}  &\qw  \\
						& \gate{S}&\gate{S} &            &\qw 
					\end{quantikz}
					=
					\begin{quantikz}
						& \gate[2]{D_3} & \gate{U}       &\qw &\qw          \\
						&            & \gate{S}  &\gate{S}   &\qw
				\end{quantikz}}
				\scalebox{0.55}{
					\begin{quantikz}
						&\qw         &\qw &   \gate[2]{D_4}  &\qw  \\
						& \gate{S}&\gate{S} &            &\qw 
					\end{quantikz}
					=
					\begin{quantikz}
						& \gate[2]{D_4} & \gate{H}       &\gate{S}&\gate{S}&\gate{H}&\qw          \\
						&            &\qw   &\qw   &\qw&\qw&\qw
				\end{quantikz}}
			\end{minipage}
			\begin{minipage}{0.4\textwidth}
				\scalebox{0.55}{
					\begin{quantikz}
						&\qw         &\qw         &\qw &\qw           &\gate[2]{D_1}  &\qw  \\
						&\gate{H} & \gate{S}& \gate{S}& \gate{H}          &&\qw 
					\end{quantikz}
					=
					\begin{quantikz}
						& \gate[2]{D_1} & \gate{S}    &\gate{S} &\qw        \\
						&            &\qw   &\qw &\qw
				\end{quantikz}}
				\scalebox{0.55}{
					\begin{quantikz}
						&\qw         &\qw         &\qw &\qw           &\gate[2]{D_2}  &\qw  \\
						&\gate{H} & \gate{S}& \gate{S}& \gate{H}          & &\qw
					\end{quantikz}
					=
					\begin{quantikz}
						& \gate[2]{D_2} & \gate{H}       &\gate{S}&\gate{S} &\gate{H}& \qw       \\
						&               &\qw                &\qw         &\qw&\qw        &\qw
				\end{quantikz}}
				\scalebox{0.55}{
					\begin{quantikz}
						&\qw         &\qw         &\qw &\qw           &\gate[2]{D_3}  &\qw  \\
						&\gate{H} & \gate{S}& \gate{S}& \gate{H}          & &\qw
					\end{quantikz}
					=
					\begin{quantikz}
						& \gate[2]{D_3} & \gate{S}       &\gate{S} &\qw        \\
						&               & \gate{S}     & \gate{S} &\qw        
				\end{quantikz}}
				\scalebox{0.55}{
					\begin{quantikz}
						&\qw         &\qw         &\qw &\qw           &\gate[2]{D_4}  &\qw  \\
						&\gate{H} & \gate{S}& \gate{S}& \gate{H}          & &\qw
					\end{quantikz}
					=
					\begin{quantikz}
						& \gate[2]{D_4} & \gate{S}       &\gate{S} &\qw        \\
						&               & \gate{S}     & \gate{S} &\qw        
				\end{quantikz}}
			\end{minipage}
		\end{center}
	\end{lemma}
\begin{proof}
The proof is analogous to that of Lemma \ref{lem:ssb}, except that we use the rewriting rules in Lemma \ref{lem:sd_sel}.
\end{proof}

	\begin{lemma}\label{lem:sb} The following rewriting rules hold:
		\begin{center}
			\begin{minipage}{0.4\textwidth}
				\scalebox{0.55}{
					\begin{quantikz}
						& \gate{S} & \gate[2]{B_1}  &\qw  \\
						&\qw       &             &\qw 
					\end{quantikz}
					=
					\begin{quantikz}
						& \gate[2]{B_1} &\qw        &\qw        &\qw        &\qw  \\
						&            & \gate{H}  & \gate{S}  &\gate{H}   &\qw  
				\end{quantikz}}
				\scalebox{0.55}{
					\begin{quantikz}
						& \gate{S} & \gate[2]{B_2}  &\qw  \\
						&\qw       &             &\qw 
					\end{quantikz}
					=
					\begin{quantikz}
						& \gate[2]{B_3} & \gate{X}       &\qw        &\qw        &\qw    &\qw  &\qw\\
						&            & \gate{S}  & \gate{S}  &\gate{S}   &  \gate{H}&\gate{S}&\qw
				\end{quantikz}}
				\scalebox{0.5}{
					\begin{quantikz}
						& \gate{S} & \gate[2]{B_3}  &\qw  \\
						&\qw       &             &\qw 
					\end{quantikz}
					=
					\begin{quantikz}
						& \gate[2]{B_2} &\qw        &\qw        &\qw        &\qw  \\
						&            & \gate{S}  & \gate{H}  &\gate{S}   &\qw  
				\end{quantikz}}
				\scalebox{0.5}{
					\begin{quantikz}
						& \gate{S} & \gate[2]{B_4}  &\qw  \\
						&\qw       &             &\qw 
					\end{quantikz}
					=
					\begin{quantikz}
						& \gate[2]{B_4} &\qw        &\qw        &\qw        &\qw  \\
						&            & \gate{H}  & \gate{S}  &\gate{H}   &\qw  
				\end{quantikz}}
			\end{minipage}
			\begin{minipage}{0.4\textwidth}
				\scalebox{0.5}{
					\begin{quantikz}
						&\gate{H}&\gate{S}& \gate{H} & \gate[2]{B_1}  &\qw  \\
						&\qw&\qw&\qw       &             &\qw 
					\end{quantikz}
					=
					\begin{quantikz}
						& \gate[2]{B_1} &\qw        &\qw          \\
						&            & \gate{S}  &\qw  
				\end{quantikz}}
				\scalebox{0.5}{
					\begin{quantikz}
						&\gate{H}&\gate{S}& \gate{H} & \gate[2]{B_2}  &\qw  \\
						&\qw&\qw&\qw       &             &\qw 
					\end{quantikz}
					=
					\begin{quantikz}
						& \gate[2]{B_2} &\qw        &\qw        &\qw        &\qw  \\
						&            & \gate{H}  & \gate{S}  &\gate{H}   &\qw  
				\end{quantikz}}
				\scalebox{0.5}{
					\begin{quantikz}
						&\gate{H}&\gate{S}& \gate{H} & \gate[2]{B_3}  &\qw  \\
						&\qw&\qw&\qw       &             &\qw 
					\end{quantikz}
					=
					\begin{quantikz}
						& \gate[2]{B_4} & \gate{X}       &\qw        &\qw          \\
						&            & \gate{S}  & \gate{S}  &\qw  
				\end{quantikz}}
				\scalebox{0.5}{
					\begin{quantikz}
						&\gate{H}&\gate{S}& \gate{H} & \gate[2]{B_4}  &\qw  \\
						&\qw&\qw&\qw       &             &\qw 
					\end{quantikz}
					=
					\begin{quantikz}
						& \gate[2]{B_3} &\qw       \\
						&            &\qw  
				\end{quantikz}}
			\end{minipage}
		\end{center}
	\end{lemma}
	\begin{proof}
		The proof is analogous to that of Lemma \ref{lem:ssb}.
	\end{proof}

	\begin{lemma}\label{lem:sd} The following rewriting rules hold:
		\begin{center}
			\begin{minipage}{0.4\textwidth}
				\scalebox{0.5}{
					\begin{quantikz}
						&\qw          &   \gate[2]{D_1}  &\qw  \\
						& \gate{S} &            &\qw 
					\end{quantikz}
					=
					\begin{quantikz}
						& \gate[2]{D_1} & \gate{H}       &\gate{S}&\gate{H}&\qw          \\
						&            &\qw   &\qw   &\qw&\qw
				\end{quantikz}}
				\scalebox{0.5}{
					\begin{quantikz}
						&\qw          &   \gate[2]{D_2}  &\qw  \\
						& \gate{S} &            &\qw 
					\end{quantikz}
					=
					\begin{quantikz}
						& \gate[2]{D_3} & \gate{S}       &\gate{S}&\gate{S}& \gate{H} &\gate{S}&\qw         \\
						&            &  \gate{S} & \gate{S}  &\qw&\qw&\qw&\qw
				\end{quantikz}}
				\scalebox{0.5}{
					\begin{quantikz}
						&\qw          &   \gate[2]{D_3}  &\qw  \\
						& \gate{S} &            &\qw 
					\end{quantikz}
					=
					\begin{quantikz}
						& \gate[2]{D_2} & \gate{S}       &\gate{H}&\gate{S}&\qw          \\
						&            &\qw   &\qw   &\qw&\qw
				\end{quantikz}}
				\scalebox{0.5}{
					\begin{quantikz}
						&\qw          &   \gate[2]{D_4}  &\qw  \\
						& \gate{S} &            &\qw 
					\end{quantikz}
					=
					\begin{quantikz}
						& \gate[2]{D_4} & \gate{H}       &\gate{S}&\gate{H}&\qw          \\
						&            &\qw   &\qw   &\qw&\qw
				\end{quantikz}}
			\end{minipage}
			\begin{minipage}{0.4\textwidth}
				\scalebox{0.5}{
					\begin{quantikz}
						&\qw         &\qw          &\qw           &\gate[2]{D_1}  &\qw  \\
						&\gate{H} & \gate{S} & \gate{H}          &\qw 
					\end{quantikz}
					=
					\begin{quantikz}
						& \gate[2]{D_1} & \gate{S}       &\qw        \\
						&            &\qw   &\qw 
				\end{quantikz}}
				\scalebox{0.5}{
					\begin{quantikz}
						&\qw         &\qw          &\qw           &\gate[2]{D_2}  &\qw  \\
						&\gate{H} & \gate{S} & \gate{H}          & &\qw
					\end{quantikz}
					=
					\begin{quantikz}
						& \gate[2]{D_2} & \gate{H}       &\gate{S} &\gate{H}&\qw        \\
						&               &\qw                &\qw         &\qw        &\qw
				\end{quantikz}}
				\scalebox{0.5}{
					\begin{quantikz}
						&\qw         &\qw          &\qw           &\gate[2]{D_3}  &\qw  \\
						&\gate{H} & \gate{S} & \gate{H}          & &\qw
					\end{quantikz}
					=
					\begin{quantikz}
						& \gate[2]{D_4} & \gate{S}       &\gate{S} &\qw        \\
						&\qw               & \gate{S}     & \gate{S} &\qw        
				\end{quantikz}}
				\scalebox{0.5}{
					\begin{quantikz}
						&\qw         &\qw          &\qw           &\gate[2]{D_4}  &\qw  \\
						&\gate{H} & \gate{S} & \gate{H}          & &\qw
					\end{quantikz}
					=
					\begin{quantikz}
						& \gate[2]{D_3} &\qw          \\
						&            &\qw 
				\end{quantikz}}
			\end{minipage}
		\end{center}
	\end{lemma}
	\begin{proof}
		The proof is analogous to that of Lemma \ref{lem:ssb}.
			\end{proof}
	
\section{The Groups $\cC_1$, $\cC_2$, and $\cC_3$}
\label{sec:C1C2}

The proofs in the subsequent sections rely on induction. Their base
cases, together with several auxiliary arguments, require explicit
computations in $\cC_1$, $\cC_2$, and $\cC_3$. The purpose of this
section is merely to collect these computational identities for later
reference. Readers may safely skip this section and refer back to it
when needed.

The identities may be verified either by direct matrix computations in
${\rm Cliff}_1$, ${\rm Cliff}_2$, and ${\rm Cliff}_3$, where equality is understood
up to a global phase, or computationally by constructing $\cC_1$,
$\cC_2$, and $\cC_3$ in GAP4 from the presentation given in
Appendix~\ref{sec:appendix}.

For $n=1$, recall that
\begin{align}
\label{eq:isomC1S4}
\cC_1
=
\langle h_1,s_1
\mid h_1^2=s_1^4=(s_1h_1)^3=1\rangle
\cong {\rm Sym}_4,
\end{align}
via 
$h_1\mapsto(12)$ and 
$
s_1\mapsto(1234)$.

For $i\in\{1,4\}$ and $j\in\{1,2\}$, define
\begin{align*}
\scalebox{0.55}{
    \begin{quantikz}
        &\gate[2]{g_{i,j}}&\qw\\
        &                  &\qw
    \end{quantikz}
}
&\coloneqq
\scalebox{0.55}{
    \begin{quantikz}
        &\gate[2]{B_i}&\gate{C_j}&\gate[2]{D_1}&\qw\\
        &             &\qw       &             &\qw
    \end{quantikz}
},\\[1ex]
\scalebox{0.55}{
    \begin{quantikz}
        &\gate[2]{k_{i,j}}&\qw\\
        &                  &\qw
    \end{quantikz}
}
&\coloneqq
\scalebox{0.55}{
    \begin{quantikz}
        &\gate[2]{B_i}&\gate{C_j}&\gate[2]{D_2}&\qw\\
        &             &\qw       &             &\qw
    \end{quantikz}
}.
\end{align*}

\begin{lemma}\label{lem:tech_1}
The following identities hold in $\cC_2$:
\[
\begin{tabular}{@{}c@{\qquad}c@{}}
$\begin{aligned}
s_1
&=
h_2s_2h_2s_2\Lambda_1h_2s_2h_2\Lambda_1h_2\Lambda_1,\\
g_{1,1}
&=
1,\\
g_{1,2}
&=
h_2s_2^2h_2,\\
g_{4,1}
&=
h_2\Lambda_1h_2,\\
g_{4,2}
&=
h_1s_1h_1s_1^3h_1
h_2\Lambda_1
h_1s_1h_1s_1^3h_1
h_2,
\end{aligned}$
&
$\begin{aligned}
\Lambda_1
&=
h_2g_{4,1}h_2,\\
k_{1,1}h_1
&=
h_2g_{4,1}h_2,\\
k_{1,2}h_1
&=
g_{1,2}h_2g_{4,1}h_2,\\
k_{4,1}h_1
&=
g_{4,1}h_2g_{4,1}h_2,\\
k_{4,2}h_1
&=
g_{4,2}h_2g_{4,1}h_2
\end{aligned}$
\end{tabular}
\]
\end{lemma}

\begin{proof}
These identities can be verified by direct matrix computations in
${\rm Cliff}_2$, where equality is understood up to a global phase.
\end{proof}

\begin{lemma}\label{lem:tech_2}
In $\cC_3$, define
\begin{align*}
M
:=
h_2\Lambda_1s_2^2h_2h_1s_1^2h_2\Lambda_1h_2,
\quad q_1
:=
s_1^3h_1\Lambda_1h_1s_1,\quad
q_2:=
s_1^3h_1\Lambda_2h_1s_1.\end{align*}
Then the following identities hold:
\begin{align*}
h_1s_1^2h_1
&=
M\Lambda_2M^{-1}\Lambda_2h_2\Lambda_2
h_3\Lambda_2h_3M\Lambda_2h_2\Lambda_2h_1h_3
M^{-1}\Lambda_2h_3h_2M^{-2},\\
&=q_1s_2h_2(q_1^{-1}s_2^{-1})^2q_1s_2(q_1h_2)^2q_1s_2h_1h_2s_2^{-1},\\
M
&=
h_2s_2^2h_2s_2q_1h_2s_2q_1,\\
q_1
&=
M\Lambda_2h_2\Lambda_2Mh_3\Lambda_2h_3
M^{-1}\Lambda_2h_2\Lambda_2h_3M^{-1}\Lambda_2h_3\Lambda_2s_2^{-1},\\
&=s_2 h_1s_1^2h_1 h_2Mh_2h_1,\\
q_2 &=
\Lambda_2.
\end{align*}
\end{lemma}

\begin{proof}
These identities can be verified by direct matrix computations in
${\rm Cliff}_3$, where equality is understood up to a global phase.
\end{proof}

\section{The centralizer of $z_1$ gate in $\cC_n$}\label{sec:Css}

Considering the inclusion 
$\cC_m\subseteq \cC_n$ for $m\leq n$, we can derive the following observation.
\begin{lemma}\label{lem:keylem}
	Let $g\in \cC_n$ and $f\in \cC_{m}$ with $m<n$, then 
	\begin{align*}
		\up Ni_{g}=\up Ni_{gf} \quad\text{for all }i\in\{n,\dots,m+1\}.
	\end{align*}
\end{lemma}
\begin{proof}
	Observe that $\up Nm_g\cdots \up N1_gf\in\cC_m$ can be rewritten into its normal form, giving us
	\begin{align*}
		gf		&=\up Nn_g\cdots \up N{m+1}\left(  \up N{m}_g\cdots \up N1_g f\right)\\
		&=\up Nn_g \cdots \up N{m+1}_g\nf\left(\up Nm_g \cdots \up N1_gf\right).
	\end{align*}
The last expression is a normal form of $gf$. By uniqueness,
its factors of levels $m+1,\ldots,n$ must coincide with those of
$g$. Hence
\begin{align*}
N_{gf}^{(i)}=N_g^{(i)}
\qquad (m+1\leq i\leq n).
\end{align*}
\end{proof}

\begin{theorem}\label{thm:|CG(ss)|}
	The order of $\cent_{\cC_n}(s_1^2)$ is given by 
	\begin{align*}
		|\cent_{\cC_n}(s_1^2)|=\begin{cases}
			8,&\text{if }n=1,\\
			\displaystyle{2^{n^2+2n} \prod_{i=1}^{n-1}(2^{2i}-1)},&\text{if }n\geq 2.
		\end{cases}
	\end{align*}
\end{theorem}
\begin{proof}
	Suppose $n=1$. Under the 
isomorphism in (\ref{eq:isomC1S4}), $s_1^2$ corresponds to the double transposition
$(13)(24)$. Since the conjugacy class of a double transposition
in ${\rm Sym}_4$ has size $3$,
\[
\left|\cent_{\cC_1}(s_1^2)\right|
=
\left|\cent_{{\rm Sym}_4}\bigl((13)(24)\bigr)\right|
=
\frac{24}{3}
=
8.
\]

	Now, we assume that $n\geq 2$.	
	Let $g \in \cent_{\cC_n}(s_1^2)$ be written in its normal form
	(\ref{eq:nf}).
	Since $s_1^2g=gs_1^2$ and $s_1^2\in\cC_1$, it follows from Lemma \ref{lem:keylem} that 
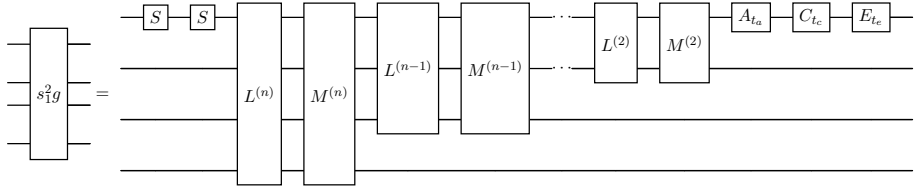
	\begin{equation*}
		\scalebox{0.55}{
			\begin{quantikz}
				&\gate[4]{s_1^2g}&\qw\\
				&				 &\qw\\
				& 		    	 &\qw\\
				&				 &\qw
			\end{quantikz}=
			\begin{quantikz}
				&\gate{S}&\gate{S}&\gate[4]{\up N{n}_g}&\gate[3]{\up N{n-1}_g}&\push{\cdots}&\gate[2]{\up N{2}_g}&\gate{\up N1_g}&\qw \\
				&\qw     &\qw     &\qw                 &\qw                   &\push{\cdots}&\qw                 &\qw            &\qw\\
				&\qw     &\qw     &\qw                 &\qw                   &\qw          &\qw                 &\qw            &\qw\\
				&\qw     &\qw     &\qw                 &\qw                   &\qw          &\qw                 &\qw            &\qw
			\end{quantikz}
		}
	\end{equation*}
	can be transformed by rewriting rules into its normal form
	\begin{equation*}
		\scalebox{0.55}{
			\begin{quantikz}
				&\gate[4]{gs_1^2}&\qw\\
				& 				 &\qw\\
				&				 &\qw\\
				&				 &\qw
			\end{quantikz}=
			\begin{quantikz}
				&\gate[4]{\up N{n}_g}&\gate[3]{\up N{n-1}_g}&\push{\cdots}&\gate[2]{\up N{2}_g}&\gate{\up N1_{gs_1^2}}&\qw \\
				&\qw                 &\qw                   &\push{\cdots}&\qw                 &\qw                   &\qw\\
				&\qw                 &\qw                   &\qw          &\qw                 &\qw                   &\qw\\
				&\qw                 &\qw                   &\qw          &\qw                 &\qw                   &\qw
			\end{quantikz}
		}.
	\end{equation*}
Hence, the $ZX$-normal forms satisfy the following preservation condition:
\begin{align}
\label{eq:preserve}
\up N i_{s_1^2g}=\up N i_g,
\quad 2\leq i\leq n.
\end{align}
	
	Let $\ell$ be the layer of $\up Nn_g$.
	Suppose $\ell=1$. Applying the rewriting rules in Lemma \ref{lem:ace_sel}, we obtain
the following possibilities:	\begin{center}
		\begin{minipage}{0.4\textwidth}
			\scalebox{0.55}{
				\begin{quantikz}
					&\gate{S}& \gate{S}&\gate{A_1} & \gate{C_i}  &  \gate[2]{D_j}&\qw\\
					&\qw&\qw       &\qw           &\qw           &               &\qw
				\end{quantikz}
				=
				\begin{quantikz}
					&\gate{A_1}& \gate{C_i} &\gate[2]{D_j}      &\qw&\qw&\qw\\
					&\qw       &\qw           &\qw              &\gate{S}         &\gate{S}&\qw
			\end{quantikz}} 
			\scalebox{0.55}{
				\begin{quantikz}
					&\gate{S}&\gate{S}&\gate{A_2}& \gate{C_1}  & \gate[2]{D_j} &\qw  \\
					&\qw       &\qw           &\qw       &\qw                &               & \qw       
				\end{quantikz}
				=
				\begin{quantikz}
					&\gate{A_2}& \gate{C_1} &\gate[2]{D_j} &\gate{U} &\qw\\
					&\qw&\qw           &               & \gate{X}  &\qw 
			\end{quantikz}} 
			\scalebox{0.55}{
				\begin{quantikz}
					&\gate{S}&\gate{S}& \gate{A_2} & \gate{C_2}  & \gate[2]{D_j} &\qw  \\
					&\qw       &\qw           &\qw          &\qw               &               &\qw
				\end{quantikz}
				=
				\begin{quantikz}
					&\gate{A_2}& \gate{C_1} &\gate[2]{D_j} &\qw\\
					&\qw&\qw           &               &\qw  
			\end{quantikz}} 
			\scalebox{0.55}{
				\begin{quantikz}
					&\gate{S}&\gate{S}&\gate{A_3}& \gate{C_{i}}  &\qw  
				\end{quantikz}
				=
				\begin{quantikz}
					&\gate{A_3}& \gate{C_{3-i}}  & \gate{S}& \gate{S}&\qw
			\end{quantikz}} 
		\end{minipage}
	\end{center}
for $i\in\{1,2\}$ and $j\in\{1,2,3,4\}$, where $U=1$ if $j=1$, and $U=h_1s_1^2h_1$ otherwise.
By condition (\ref{eq:preserve}), the third case is impossible,
since the rewriting changes the $C$-block from $C_2$ to $C_1$.

In the first and fourth cases, two trailing $S$-gates are produced
on the second wire. Since the $D$-blocks occur successively on
adjacent pairs of wires, repeated application of
Lemma \ref{lem:sd_propagation} moves these gates downward, one wire
at a time, until they reach the $E$-block of $\up Nn_g$.
By the $E$-rewriting rules in Lemma \ref{lem:ace_sel}, each
application of $S$ shifts $E_i$ to $E_{i+1}$, with the indices taken
modulo $4$. Hence the two $S$-gates change the $E$-block, contradicting
condition (\ref{eq:preserve}). Therefore, the first and fourth cases
are impossible.

It remains to consider the second case. In this case, the rewriting
produces a trailing $X$-gate on the second wire. Direct matrix
computation in ${\rm Cliff}_2$ gives
\begin{align*}
X_1D_1=D_1X_2,\quad X_1D_j=D_jX_1X_2,
\quad j\in\{2,3,4\}.
\end{align*}
Thus, as the $D$-blocks occur successively on adjacent pairs of wires,
a nontrivial $X$-gate is passed to the next wire at each step and
eventually reaches the $E$-block of $\up Nn_g$.

Since
\[
\{E_1,E_2,E_3,E_4\}=\langle s_1\rangle,
\]
the resulting terminal factor is of the form $X_1E_i$. However,
\[
X_1E_i\notin\langle s_1\rangle
=
\{E_1,E_2,E_3,E_4\}.
\]
Hence the resulting circuit cannot have the same $n$-th normal-form
factor as $\up Nn_g$, contradicting condition~\eqref{eq:preserve}.
Therefore, the second case is also impossible, and thus $\ell\neq1$.

	Now, we suppose that $\ell>1$. Again, we have $\up N n_{s_1^2g}=\up N n_{g}$.
Applying Lemma \ref{lem:ssb}, we first eliminate the cases
$B_2$ and $B_3$, since in these cases the rewriting changes the
$C$-block and therefore violates condition (\ref{eq:preserve}).

For the remaining cases, Lemma \ref{lem:ssd} shows that the cases
$D_3$ and $D_4$ produce two $S$-gates on the second wire when
$h_2s_2^2h_2$ is moved through the top $D$-block. These $S$-gates
then propagate through the subsequent blocks until they reach the
$E$-block of $\up Nn_g$, thereby changing $\up Nn_g$ and
contradicting condition (\ref{eq:preserve}).

Therefore, the top basic gates in $\up Nn_g$ must satisfy one of the
following conditions:	\begin{equation}\label{eq:typssg=gss}
		\scalebox{0.55}{
			\begin{quantikz}
				&\gate{S}& \gate{S}&\gate[2]{B} & \gate{C} &\gate[2]{D_1}  &\qw  \\
				&\qw       &\qw          &           &\qw           &               &\qw
			\end{quantikz}
			=
			\begin{quantikz}
				&\gate[2]{B}& \gate{C}&\gate[2]{D_1} & \gate{S} &\gate{S}  &\qw  \\
				&       &\qw          &           &\qw           &\qw               &\qw
			\end{quantikz}
		},
	\end{equation}
	\begin{equation}\label{eq:typhsshg=gss}
		\scalebox{0.55}{
			\begin{quantikz}
				&\gate{S}& \gate{S}&\gate[2]{B} & \gate{C} &\gate[2]{D_2}  &\qw  \\
				&\qw       &\qw          &           & \qw          &               &\qw
			\end{quantikz}
			=
			\begin{quantikz}
				&\gate[2]{B}& \gate{C}&\gate[2]{D_2} & \gate{H}&\gate{S} &\gate{S}  &\gate{H}&\qw  \\
				&       &\qw          &           &\qw           &\qw               &\qw&\qw&\qw
			\end{quantikz}
		},
	\end{equation}
	where $B\in \{B_1,B_4\}$.
	
	If they satisfy (\ref{eq:typssg=gss}), then
	\begin{align*}
		s_1^2	\left(\up N{n-1}_g\cdots \up N1_g\right)=\left(\up N{n-1}_g\cdots \up N1_g\right)s_1^2.
	\end{align*}
	On the other hand, if they satisfy  (\ref{eq:typhsshg=gss}), then \begin{align*}
		h_1s_1^2h_1 \left(\up N{n-1}_g\cdots \up N1_g\right)=\left(\up N{n-1}_g\cdots \up N1_g\right)s_1^2.
	\end{align*} 
	By setting $f = h_1\up N{n-1}_g\cdots \up N1_g h_1$, the above formula becomes
	\begin{align*}
		s_1^2(fh_1)=(fh_1)s_1^2.
	\end{align*}
	We can thus conclude that
	\begin{align}\label{eq:sizeequals1^2}
		\left|\left\{ u\in \cC_{n-1}: s_1^2u=us_1^2  \right\} \right|=	\left|\left\{ u \in \cC_{n-1}: h_1s_1^2h_1u =us_1^2  \right\} \right|.
	\end{align}
We now prove by induction on $m\geq1$ that
\begin{align*}
|\cent_{\cC_m}(s_1^2)|
=
2^{m^2+2m}
\prod_{i=1}^{m-1}(2^{2i}-1),
\end{align*}
where the empty product is understood to be $1$.
	
The case $m=1$ was established above. Now let $m\geq2$ and assume
that
\begin{align*}
|\cent_{\cC_{m-1}}(s_1^2)|
=
2^{(m-1)^2+2(m-1)}
\prod_{i=1}^{m-2}(2^{2i}-1).
\end{align*}

By the preceding analysis, the top basic gates of $\up Nm_g$
must satisfy either \eqref{eq:typssg=gss} or
\eqref{eq:typhsshg=gss}. For each of these two cases, there are
\[
3\sum_{k=0}^{m-2}4^k
\]
choices for the layer of $\up Nm_g$, the $A$-block, and the
preceding unrestricted $B$-blocks. There are $2$ choices for the
top $B$-block, $2$ choices for the $C$-block, and $4^{m-1}$
choices for the remaining $D$-blocks and the $E$-block.

By \eqref{eq:sizeequals1^2}, the two cases admit the same number of
choices for the lower normal-form factors. In each case, this number
is $
|\cent_{\cC_{m-1}}(s_1^2)|$.
Therefore, using the induction hypothesis, we obtain
\begin{align*}
\left|\cent_{\cC_m}(s_1^2)\right|
&=
3\left(\sum_{k=0}^{m-2}4^k\right)
\cdot2\cdot2\cdot4^{m-1}\cdot2
\left|\cent_{\cC_{m-1}}(s_1^2)\right|\\
&=
3\left(\sum_{k=0}^{m-2}4^k\right)
\cdot2\cdot2\cdot4^{m-1}\cdot2
\left(
2^{(m-1)^2+2(m-1)}
\prod_{i=1}^{m-2}(2^{2i}-1)
\right)\\
&=\left(2^{2(m-1)}-1 \right)
\cdot2\cdot2\cdot4^{m-1}\cdot2
\left(
2^{(m-1)^2+2(m-1)}
\prod_{i=1}^{m-2}(2^{2i}-1)
\right)\\
&=
2^{m^2+2m}
\prod_{i=1}^{m-1}(2^{2i}-1).
\end{align*}
This completes the induction.
\end{proof}

In the remainder of this section, we will prove the following theorem:
\begin{theorem}\label{thm:CG(s1^2)} The centralizer of $s_1^2$ in $\cC_n$ is given by
	\begin{align*}
		\cent_{\cC_n}(s_1^2) =\left\{
		\begin{array}{ll}
			\langle h_1s_1h_1s_1^3h_1, s_1\rangle,&\text{if }n=1,\\
			\geng{h_1s_1h_1s_1^3h_1,h_2,\dots,h_n}{s_1,s_2,\dots,s_n}{\Lambda_1,\Lambda_2,\dots,\Lambda_{n-1}},&\text{if }n\geq 2.
		\end{array}\right.
	\end{align*}
	If $n\geq 2$, then the generator $s_1$ can be omitted.
\end{theorem}

\begin{proof}
Suppose first that $n=1$. Under the isomorphism
(\ref{eq:isomC1S4}), we have
\begin{align*}
s_1^2\mapsto(13)(24),\quad
h_1s_1h_1s_1^3h_1&\mapsto(24).
\end{align*}
Since
$
\cent_{{\rm Sym}_4}\bigl((13)(24)\bigr)
=
\langle (1234),(24)\rangle
$,
it follows that
\begin{align*}
\cent_{\cC_1}(s_1^2)
=
\left\langle
s_1,\,
h_1s_1h_1s_1^3h_1
\right\rangle.
\end{align*}

Now assume that $n\geq2$. By Lemma~\ref{lem:tech_1}, the generator
$s_1$ may be omitted, since
\begin{align}
\label{eq:s1}
s_1
=
h_2s_2h_2s_2\Lambda_1h_2s_2h_2\Lambda_1h_2\Lambda_1.
\end{align}
	Let 
	\begin{equation*}
		\scalebox{0.55}{
			\begin{quantikz}
				&\gate[2]{g_{i,j}}&\qw\\
				&                   &\qw
			\end{quantikz}
		}:=
		\scalebox{0.55}{
			\begin{quantikz}
				&\gate[2]{B_i}&\gate{C_j}&\gate[2]{D_1}&\qw\\
				&             &\qw       &             &\qw
			\end{quantikz}
		}
	\end{equation*}
	for all $i\in\{1,4\}$ and $j\in\{1,2\}$.
	
	Let $G_n$ be the subgroup given in the statement, and define $F_n$ as
	\begin{align*}
		F_n =\begin{cases}
			\displaystyle{ \left\langle
				\begin{matrix}
					g_{1,1},g_{1,2},g_{4,1},g_{4,2},h_1s_1h_1s_1^3h_1, h_2,\\
					s_2
				\end{matrix}\right\rangle },&\text{if }n= 2,\\	
			\displaystyle{  \geng{g_{1,1},g_{1,2},g_{4,1},g_{4,2},h_1s_1h_1s_1^3h_1, h_2,\dots,h_n }{s_2,s_3,\dots,s_n}{\Lambda_2,\dots,\Lambda_{n-1}}},&\text{if }n\geq 3.			
		\end{cases}
	\end{align*}	
By Lemma \ref{lem:tech_1}, the generators of $F_n$ can be expressed
in terms of the generators of $G_n$, and conversely. Hence,
$
F_n=G_n.
$
Moreover,
$
G_n\subseteq\cent_{\cC_n}(s_1^2),
$
since every generator of $G_n$ commutes with $s_1^2$.
It therefore remains to prove, by induction on $n$, that
\begin{align*}
\cent_{\cC_n}(s_1^2)\subseteq F_n.
\end{align*}

We first consider the case $n=2$. 
Let
$
g=\up N2_g\up N1_g\in\cent_{\cC_2}(s_1^2)
$.
Recall that the top basic gates of $\up N2_g$ must satisfy either
(\ref{eq:typssg=gss}) or (\ref{eq:typhsshg=gss}).
If they satisfy (\ref{eq:typssg=gss}), then
$
\up N1_g\in\cent_{\cC_1}(s_1^2)\subseteq F_2
$.
Since $\up N2_g\in F_2$, it follows that $g\in F_2$.
Suppose instead that they satisfy (\ref{eq:typhsshg=gss}). Then
\begin{align*}
h_1s_1^2h_1\up N1_g
=
\up N1_gs_1^2.
\end{align*}
Multiplying both sides on the left by $h_1$, we obtain
\begin{align*}
s_1^2\left(h_1\up N1_g\right)
=
\left(h_1\up N1_g\right)s_1^2.
\end{align*}
Thus,
$
h_1\up N1_g
\in
\cent_{\cC_1}(s_1^2)
\subseteq
F_2
$.
By Lemma \ref{lem:tech_1}, we also have
\begin{equation}
\label{eq:BCD2h_1}
\scalebox{0.55}{
\begin{quantikz}
&\gate[2]{B}&\gate{C}&\gate[2]{D_2}&\gate{H}&\qw\\
&           &\qw     &             &\qw     &\qw
\end{quantikz}
}
\in F_2,
\end{equation}
where $B\in\{B_1,B_4\}$. Therefore,
\begin{align*}
g
=
\left(\up N2_gh_1\right)
\left(h_1\up N1_g\right)
\in
F_2.
\end{align*}
	Hence, $\cent_{\cC_2 }(s_1^2) =G_2$.

Now assume that $n\geq3$ and that the result holds for $n-1$.
Let $
g=\up Nn_g\up N{n-1}_g\cdots\up N1_g\in\cent_{\cC_n}(s_1^2)
$.
Recall that the top basic gates of $\up Nn_g$ must satisfy either
(\ref{eq:typssg=gss}) or (\ref{eq:typhsshg=gss}).

Suppose first that they satisfy (\ref{eq:typssg=gss}). By the induction hypothesis,
\begin{align*}
\up N{n-1}_g\cdots\up N1_g
\in
\cent_{\cC_{n-1}}(s_1^2) \subseteq  F_{n-1}\subseteq F_n.
\end{align*}
Moreover, the form of $\up Nn_g$ shows that $\up Nn_g\in F_n$.
Therefore,
\begin{align*}
g
=
\up Nn_g
\left(\up N{n-1}_g\cdots\up N1_g\right)
\in
F_n.
\end{align*}

Suppose instead that they satisfy
(\ref{eq:typhsshg=gss}). Then
\begin{align*}
g
=
\left(\up Nn_gh_1\right)
\left(h_1\up N{n-1}_g\cdots\up N1_g\right).
\end{align*}
By Lemma \ref{lem:tech_1},
$
\up Nn_gh_1\in F_n
\subseteq
\cent_{\cC_n}(s_1^2).
$
Since both $g$ and $\up Nn_gh_1$ commute with $s_1^2$, we have
\begin{align*}
h_1\up N{n-1}_g\cdots\up N1_g
&=
\left(\up Nn_gh_1\right)^{-1}g
\in
\cent_{\cC_n}(s_1^2) \cap \cC_{n-1} =\cent_{\cC_{n-1}}(s_1^2) .
\end{align*}
By the induction hypothesis,
$
h_1\up N{n-1}_g\cdots\up N1_g
\in
F_{n-1}
\subseteq
F_n.
$
Consequently,
$
g
\in
F_n
$
and
\begin{align*}
\cent_{\cC_n}(s_1^2)\subseteq F_n.
\end{align*}
This completes the induction.
\end{proof}

\section{The centralizer of $s_1$ gate in $\cC_n$}\label{sec:Cs}

\begin{theorem}\label{thm:|CG(s1)|}
	The order of $\cent_{\cC_n}(s_1)$ is given by
	\begin{align*}
		|\cent_{\cC_n}(s_1)|=\begin{cases}
			4,&\text{if }n=1,\\
			\displaystyle{2^{n^2+2n-1} \prod_{i=1}^{n-1}(2^{2i}-1)},&\text{if }n\geq 2.
		\end{cases}
	\end{align*}
\end{theorem}
\begin{proof}
For $n=1$, under the isomorphism in (\ref{eq:isomC1S4}), the element
$s_1$ corresponds to a $4$-cycle. Since the centralizer of a
$4$-cycle in ${\rm Sym}_4$ is the cyclic subgroup generated by that
cycle, we have
\begin{align}\label{eq:C1(s1)=<s1>}
\cent_{\cC_1}(s_1)=\langle s_1\rangle
\end{align}
and 
$
|\cent_{\cC_1}(s_1)|=4.
$

Now, we assume that $n\geq 2$.
	Let $g \in \cent_{\cC_n}(s_1)$ be written in its normal form (\ref{eq:nf}).
	Since $s_1g=gs_1$ and $s_1\in\cC_1$, it follows from Lemma \ref{lem:keylem} that 
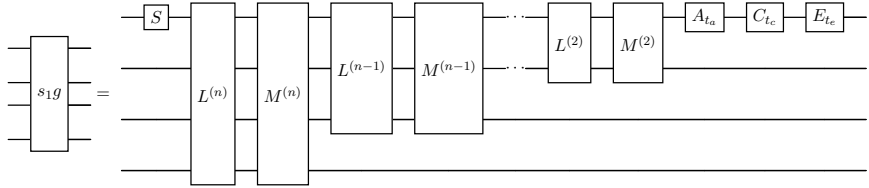
	\begin{equation*}
		\scalebox{0.55}{
			\begin{quantikz}
				&\gate[4]{s_1g} &\qw\\
				& &\qw\\
				& &\qw\\
				& &\qw
			\end{quantikz}=
			\begin{quantikz}
				&\gate{S}&\gate[4]{\up N{n}_g}&\gate[3]{\up N{n-1}_g}&\push{\cdots}&\gate[2]{\up N{2}_g}&\gate{\up N1_g}&\qw\\
				&\qw     &\qw                 &\qw                   &\push{\cdots}&\qw                 &\qw            &\qw\\
				&\qw     &\qw                 &\qw                   &\qw          &\qw                 &\qw            &\qw\\
				&\qw     &\qw                  &\qw                  &\qw          &\qw                 &\qw            &\qw
			\end{quantikz}
		}
	\end{equation*}
	can be transformed by rewriting rules into its normal form
	\begin{equation*}
		\scalebox{0.55}{
			\begin{quantikz}
				&\gate[4]{gs_1} &\qw\\
				& &\qw\\
				& &\qw\\
				& &\qw
			\end{quantikz}=
			\begin{quantikz}
				&\gate[4]{\up N{n}_g}&\gate[3]{\up N{n-1}_g}&\push{\cdots}&\gate[2]{\up L{2}_g}&\gate{\up L1_{gs_1}}&\qw\\
				&\qw                 &\qw                   &\push{\cdots}&\qw                 &\qw                 &\qw\\
				&\qw                 &\qw                   &\qw          &\qw                 &\qw                 &\qw\\
				&\qw                 &\qw                   &\qw          &\qw                 &\qw                 &\qw
			\end{quantikz}
		}.
	\end{equation*}
	Hence, the $ZX$-normal forms satisfy the following preservation condition:
	\begin{align} \label{eq:preserve_s}
	\up N i_{s_1g} =\up N i_g ,\quad (2\leq i\leq n).
	\end{align}

	Since $g\in \cent_{\cC_n}(s_1)\subseteq \cent_{\cC_n}(s_1^2)$, the layer of $\up Nn_g$ cannot be 1.
Now, suppose that $\ell>1$. Applying Lemma \ref{lem:sb}, we first
eliminate the cases $B_2$ and $B_3$, since in these cases the
rewriting changes either the $B$-block or the $C$-block and therefore
violates (\ref{eq:preserve_s}).

For the remaining cases, Lemma \ref{lem:sd} eliminates both $D_3$
and $D_4$. In the case $D_3$, moving $h_2s_2h_2$ through the top
$D$-block produces two $S$-gates on the second wire. These $S$-gates then
propagate through the subsequent blocks until it reaches the
$E$-block of $\up Nn_g$, thereby changing $\up Nn_g$. In the case
$D_4$, the rewriting changes the top $D$-block itself. Thus, both
cases contradict condition (\ref{eq:preserve_s}).

Therefore, the top basic gates in $\up Nn_g$ must satisfy one of the following conditions:
	\begin{equation}\label{eq:typsg=gs}
		\scalebox{0.55}{
			\begin{quantikz}
				& \gate{S}&\gate[2]{B} & \gate{C} &\gate[2]{D_1}  &\qw  \\
				&\qw         &           &\qw           &               &\qw
			\end{quantikz}
			=
			\begin{quantikz}
				&\gate[2]{B}& \gate{C}&\gate[2]{D_1} & \gate{S}   &\qw  \\
				&       &\qw          &           &\qw                   &\qw
			\end{quantikz}
		},
	\end{equation}
	\begin{equation}\label{eq:typhshg=gs}
		\scalebox{0.55}{
			\begin{quantikz}
				& \gate{S}&\gate[2]{B} & \gate{C} &\gate[2]{D_2}  &\qw  \\
				&\qw         &           &\qw           &               &\qw
			\end{quantikz}
			=
			\begin{quantikz}
				&\gate[2]{B}& \gate{C}&\gate[2]{D_2} & \gate{H}&\gate{S}   &\gate{H}&\qw  \\
				&       &\qw          &           & \qw          &\qw               &\qw&\qw
			\end{quantikz}
		},
	\end{equation}
	where $B\in \{B_1,B_4\}$.

	If they satisfy (\ref{eq:typsg=gs}), then
	\begin{align*}
		s_1	\left(\up N{n-1}_g\cdots \up N1_g\right)=	\left(\up N{n-1}_g\cdots \up N1_g\right)s_1.
	\end{align*}
	On the other hand, if they satisfy  (\ref{eq:typhshg=gs}), then \begin{align*}
		h_1s_1h_1 	\left(\up N{n-1}_g\cdots \up N1_g\right)=	\left(\up N{n-1}_g\cdots \up N1_g\right)s_1.
	\end{align*} 
	By setting $f = h_1  \up N{n-1}_g\cdots \up N1_g h_1$, the above formula becomes
	\begin{align*}
		s_1(fh_1)=(fh_1)s_1.
	\end{align*}
	We can thus conclude that
	\begin{align}\label{eq:sizeequals1}
		\left|\left\{ u\in \cC_{n-1}: s_1u=us_1  \right\} \right|=	\left|\left\{ u \in \cC_{n-1}: h_1s_1h_1u =us_1  \right\} \right|.
	\end{align}
	
	We now prove by induction on $m\geq1$ that
\begin{align*}
|\cent_{\cC_m}(s_1)|
=
2^{m^2+2m-1}
\prod_{i=1}^{m-1}(2^{2i}-1),
\end{align*}
where the empty product is understood to be $1$.

The case $m=1$ was established above. Now let $m\geq2$ and assume
that
\begin{align*}
|\cent_{\cC_{m-1}}(s_1)|
=
2^{(m-1)^2+2(m-1)-1}
\prod_{i=1}^{m-2}(2^{2i}-1).
\end{align*}

By the preceding analysis, the top basic gates of $\up Nm_g$
must satisfy either (\ref{eq:typsg=gs}) or (\ref{eq:typhshg=gs}).
For each of these two cases, there are
\begin{align*}
3\sum_{k=0}^{m-2}4^k
\end{align*}
choices for the layer of $\up Nm_g$, the $A$-block, and the preceding unrestricted $B$-blocks. There are $2$ choices for
the top $B$-block, $2$ choices for the $C$-block, and
$
4^{m-1}
$
choices for the remaining $D$-blocks and the $E$-block.

By (\ref{eq:sizeequals1}), the two cases admit the same number of choices for the lower normal-form factors. In each case, this number is 
\begin{align*}
|\cent_{\cC_{m-1}}(s_1)|.
\end{align*}
Therefore, using the induction hypothesis, we obtain
\begin{align*}
|\cent_{\cC_m}(s_1)|
&=
3\left(\sum_{k=0}^{m-2}4^k\right)
\cdot2\cdot2\cdot4^{m-1}\cdot2\,
|\cent_{\cC_{m-1}}(s_1)|\\
&=
3\left(\sum_{k=0}^{m-2}4^k\right)
\cdot2\cdot2\cdot4^{m-1}\cdot2
\left(
2^{(m-1)^2+2(m-1)-1}
\prod_{i=1}^{m-2}(2^{2i}-1)
\right)\\
&=
2^{m^2+2m-1}
\prod_{i=1}^{m-1}(2^{2i}-1).
\end{align*}
This completes the induction.
\end{proof}

\begin{theorem}\label{thm:CG(s1)} The centralizer of $s_1$ in $\cC_n$ is given by
	\begin{align*}
		\cent_{\cC_n}(s_1) =\left\{
		\begin{array}{ll}
			\langle s_1\rangle,&\text{if }n=1,\\
			\geng{h_2,\dots,h_n}{s_1,s_2,\dots,s_n}{\Lambda_1,\Lambda_2,\dots,\Lambda_{n-1}},&\text{if }n\geq 2.
		\end{array}\right.
	\end{align*}
	Moreover, the $\cent_{\cC_n}(s_1)$ is a normal subgroup of $\cent_{\cC_n}(s_1^2)$.
	If $n\geq 2$, then the generator $s_1$ can be omitted.
\end{theorem}
\begin{proof}
The case $n=1$ was already established in
(\ref{eq:C1(s1)=<s1>}).
Now, we assume $n\geq 2$.
	The generator $s_1$ can be omitted since (\ref{eq:s1}).
	
	We first observe that 
	\begin{align*}
		F_n:=	 \geng{h_2,\dots,h_n}{s_2,\dots,s_n}{\Lambda_1,\Lambda_2,\dots,\Lambda_{n-1}}\subseteq \cent_{\cC_n}(s_1) .
	\end{align*}
	Thus, it suffices to show that the order of $\cent_{\cC_n}(s_1)$ is equal to the order of $F_n$.
	
	The only generator of $\cent_{\cC_n}(s_1^2)$ not in $F_n$ is $h_1s_1h_1s_1^3h_1$. By
	\begin{align*}
		(h_1s_1h_1s_1^3h_1)^{-1}\Lambda_1 (h_1s_1h_1s_1^3h_1) &=\Lambda_1s_2^2,\\
		(h_1s_1h_1s_1^3h_1)^{-1}\Lambda_2 (h_1s_1h_1s_1^3h_1) &=\Lambda_2,
	\end{align*}
	it follows that $F_n$ is normal in $\cent_{\cC_n}(s_1^2)$, and the order of $\cent_{\cC_n}(s_1^2)/F_n$ is $2$.
	By Theorem \ref{thm:|CG(s1)|}, the orders of $F_n$ and $\cent_{\cC_n}(s_1)$ are the same, completing the proof.
\end{proof}

\subsection{Application: A permutation representation of $\cC_n$}\label{sec:perm}
The subgroup 
\begin{align*}
	\langle\sigma_{(1,2)},\dots, \sigma_{(n-1,n)} \rangle
\end{align*} of $\cC_n$ is isomorphic to the symmetric group ${\rm Sym}_n$ via the mapping $\sigma_{(i,i+1)}$ to $(i,i+1)$. 
Hence, we may denote $\sigma_p \in \cC_n$ for the corresponding element $p \in {\rm Sym}_n$. 
The action of $\sigma_p$ on $h_i,s_i, \Lambda_{i\to j}$ are
\begin{align}
	\sigma_p^{-1}h_i\sigma_p =h_{ p(i) },\quad \sigma_p^{-1} s_i \sigma_p= s_{p(i)},\quad  \sigma_p^{-1}\Lambda_{i\to j} \sigma_p= \Lambda_{p(i)\to p(j)} \label{eq:permaction}
\end{align}
where $p(i)$ is the image of $i$ under the permutation $p$, and $\Lambda_{i \to j}$ denotes the controlled $z$ gate with qubit $i$ as the control and qubit $j$ as the target.

\begin{lemma}\label{lem:C_G(s)capC_G(hsh)}  The centralizer of $\{s_1,h_1s_1h_1\}$ in $\cC_n$ is given by
	\begin{align*}
		\cent_{\cC_n} (\{s_1,h_1s_1h_1\}) =\left\{ 
		\begin{array}{ll}
			\langle 1\rangle,&\text{if }n=1,\\
			\langle h_2,s_2\rangle,&\text{if }n=2,\\
			\geng{ h_{2},\dots, h_n}{s_{2},\dots, s_n}{\Lambda_{2},\dots, \Lambda_{n-1}}, &\text{if }n\geq 3.\end{array}
		\right.
	\end{align*}
\end{lemma}
\begin{proof}
For $n=1$, recall that
$
\cent_{\cC_1}(s_1)=\langle s_1\rangle.
$
A direct computation shows that $h_1s_1h_1$ does not commute with
$s_1^i$ for $i\in\{1,2,3\}$. Hence,
\begin{align*}
\cent_{\cC_1}\bigl(\{s_1,h_1s_1h_1\}\bigr)
=
\langle 1\rangle.
\end{align*}

 We first note that
$
\langle h_2,s_2\rangle
\subseteq
\cent_{\cC_2}(\{s_1,h_1s_1h_1\}),
$
and, for $n\geq3$,
\begin{align*}
\geng{h_2,\dots,h_n}{s_2,\dots,s_n}{\Lambda_2,\dots,\Lambda_{n-1}}
\subseteq
\cent_{\cC_n}(\{s_1,h_1s_1h_1\}).
\end{align*}
	Since the left-hand side is isomorphic to $\cC_{n-1}$, it suffices to show that the order of the right-hand side is equal to $|\cC_{n-1}|$, where we adopt the convention that $\cC_0=\langle1\rangle$.
	
We prove the statement by induction on $n$. The case $n=1$ has
already been established.
Now assume that $n\geq 2$.

	Let $g \in \cent_{\cC_n}(\{s_1,h_1s_1h_1\})$ be written in its normal form (\ref{eq:nf}).
	Since $h_1s_1h_1g=gh_1s_1h_1$ and $h_1s_1h_1\in\cC_1$, it follows from Lemma \ref{lem:keylem} that 	
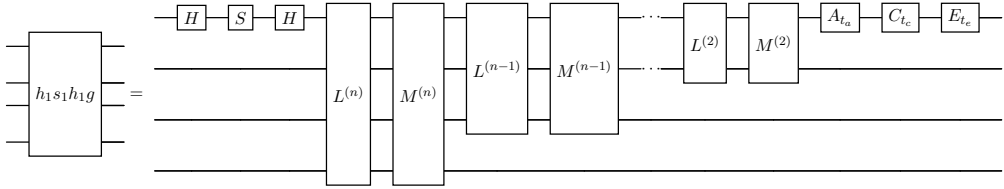
	\begin{equation*}
		\scalebox{0.55}{
			\begin{quantikz}
				&\gate[4]{h_1s_1h_1g}&\qw\\
				&				 &\qw\\
				& 		    	 &\qw\\
				&				 &\qw
			\end{quantikz}=
			\begin{quantikz}
				&\gate{H}&\gate{S}&\gate{H}&\gate[4]{\up N{n}_g}&\gate[3]{\up N{n-1}_g}&\push{\cdots}&\gate[2]{\up N{2}_g}&\gate{\up N1_g}&\qw\\
				&\qw     &\qw     &\qw     &\qw                 &\qw                   &\push{\cdots}&\qw                 &\qw            &\qw\\
				&\qw     &\qw     &\qw     &\qw                 &\qw                   &\qw          &\qw                 &\qw            &\qw\\
				&\qw     &\qw     &\qw     &\qw                 &\qw                   &\qw          &\qw                  &\qw           &\qw
			\end{quantikz}
		}
	\end{equation*}
	can be transformed by rewriting rules into its normal form
	\begin{equation*}
		\scalebox{0.55}{
			\begin{quantikz}
				&\gate[4]{gh_1s_1h_1} &\qw\\
				& &\qw\\
				& &\qw\\
				& &\qw
			\end{quantikz}=
			\begin{quantikz}
				&\gate[4]{\up N{n}_g}&\gate[3]{\up N{n-1}_g}&\push{\cdots}&\gate[2]{\up N{2}_g} &\gate{\up N1_{gh_1s_1h_1}}&\qw\\
				&\qw                 &\qw                   &\push{\cdots}&\qw                  &\qw                       &\qw\\
				&\qw                 &\qw                   &\qw          &\qw                  &\qw                       &\qw\\
				&\qw                 &\qw                   &\qw          &\qw                  &\qw                       &\qw
			\end{quantikz}
		}.
	\end{equation*}
	Hence, the $ZX$-normal forms satisfy the following preservation condition:
	\begin{align}\label{eq:preserve_h1s1h1}
	\up Ni_{h_1s_1h_1g}=\up Ni_{g} ,\quad 2\leq i\leq n.
	\end{align}

	Recall that the layer of $\up Nn_g$ is greater than $1$ since $g \in \cent_{\cC_n}(s_1)$.
	Moreover, the equations (\ref{eq:typsg=gs}) and (\ref{eq:typhshg=gs}) hold. 
	Together with the rewriting rules in Lemmas \ref{lem:sb} and \ref{lem:sd}, these imply that the top basic
gates of $\up Nn_g$ must satisfy the following two conditions in order
for the rewriting from $h_1s_1h_1g$ to $gh_1s_1h_1$ to preserve the
normal form:
	\begin{equation*}
		\scalebox{0.55}{
			\begin{quantikz}
				& \gate{S}&\gate[2]{B_1} & \gate{C} &\gate[2]{D_1}  &\qw  \\
				&\qw         &\qw           &\qw           &               &\qw
			\end{quantikz}
			=
			\begin{quantikz}
				&\gate[2]{B_1}& \gate{C}&\gate[2]{D_1}&\gate{S}  &\qw  \\
				&       &\qw          &           &\qw                         &\qw
			\end{quantikz}
		},
	\end{equation*}
	\begin{equation*}
		\scalebox{0.55}{
			\begin{quantikz}
				&\gate{H}&\gate{S}& \gate{H}&\gate[2]{B_1} & \gate{C} &\gate[2]{D_1}  &\qw  \\
				&\qw     &\qw     &\qw      &           &\qw           &               &\qw
			\end{quantikz}
			=
			\begin{quantikz}
				&\gate[2]{B_1}& \gate{C}&\gate[2]{D_1} & \gate{H}&\gate{S}   &\gate{H}&\qw  \\
				&             &\qw          &           &\qw           &\qw               &\qw&\qw
			\end{quantikz}
		},
	\end{equation*}
	Thus, $\up N{n-1}_g\cdots \up N1_g \in \cent_{\cC_{n-1}}(\{s_1,h_1s_1h_1\})$.
	Since other basic gates in $\up Nn_g$ can be chosen without constraint, it follows from the induction hypothesis and (\ref{eq:orderCn}) that
	\begin{align*}
		|\cent_{\cC_{n}}(\{s_1,h_1s_1h_1\}) | &= 3 \left( \sum_{k=0}^{n-2} 4^k \right) \cdot 2\cdot 4^{n-1} \cdot|\cent_{\cC_{n-1}}(\{s_1,h_1s_1h_1\})|\\
		&= 3 \left( \sum_{k=0}^{n-2} 4^k \right) \cdot 2\cdot 4^{n-1} \cdot|\cC_{n-2}|=|\cC_{n-1}|.
	\end{align*}
	The proof is completed by induction.
\end{proof}

We are now ready for the main theorem of this paper.
\begin{theorem}
	Let $V$ be the conjugacy class of $s_1$ in $\cC_n$ and let $\sigma_g$ denote the permutation induced by the conjugate action of $g \in \cC_n$ on $V$. Then the map $\cC_n \to {\rm Sym}_{|V|}: g \mapsto \sigma_g$ is a faithful homomorphism. Moreover, the size of $V$ is equal to $2(4^n - 1)$.
\end{theorem}
\begin{proof}
	It is clear that the map $\cC_n\to {\rm Sym}_{|V|}: g \mapsto \sigma_g$ is a homomorphism.
	
	Suppose this map is not faithful. Then there exist $\sigma_f=\sigma_g$ for some distinct $f,g\in \cC_n$ such that
	\begin{align*}
		f^{-1}v f=g^{-1}vg \text{ for all }v \in V.
	\end{align*}
	This is equivalent to 
	\begin{align*}
		(gf^{-1}) v= v (gf^{-1}) \text{ for all }v \in V.
	\end{align*}
	It follows from (\ref{eq:permaction}) that  $s_i=\sigma_{(1,i)}s_1\sigma_{(1,i)}$ for $i\in\{2,\dots,n\}$.
	Thus, we conclude that $s_1,\dots,s_n$ and $h_1 s_1 h_1, \dots, h_n s_n h_n$ are all contained in $V$.
	Note that  
	\begin{align}\label{eq:capCsChsh}
		\bigcap_{v \in V }  \cent_{\cC_n} (v) \subseteq  \bigcap_{k=1 } ^n  \cent_{\cC_n} (s_k) \cap \cent_{\cC_n} (h_ks_kh_k). 
	\end{align}
	Let $p$ be the permutation $(1,n,n-1,\dots,2)$ and let $G_k = \cent_{\cC_n}(\{s_k,h_k s_k h_k\})$ for all $k \in \{1, \dots, n\}$.
	By (\ref{eq:permaction}) and Lemma \ref{lem:C_G(s)capC_G(hsh)}, we have
	\begin{align*}
		\sigma_p^{-1}G_1\sigma_p=  \sigma_p^{-1}\geng{h_2,\dots,h_n}{s_2,\dots,s_n}{\Lambda_2,\dots,\Lambda_{n-1}}\sigma_p= \geng{h_1,\dots, h_{n-1}}{s_1,\dots,s_{n-1}}{\Lambda_{1},\dots,\Lambda_{n-2}}=\cC_{n-1}.
	\end{align*}
	It immediately follows that 
	\begin{align*} 
		\sigma_p^{-1}G_1\sigma_p \cap 	\sigma_p^{-1}\cent_{\cC_n}(s_2)\sigma_p&= \cC_{n-1}\cap \cent_{\cC_n}(s_1) =\cent_{\cC_{n-1}}(s_1) .\\
		\sigma_p^{-1}G_1\sigma_p \cap 	\sigma_p^{-1}\cent_{\cC_n}(h_2s_2h_2)\sigma_p&= \cC_{n-1}\cap \cent_{\cC_n}(h_1s_1h_1) =\cent_{\cC_{n-1}}(h_1s_1h_1) .
	\end{align*}
	Applying Lemma \ref{lem:C_G(s)capC_G(hsh)} again yields
	\begin{align*}
		\bigcap_{k=1}^2 G_k =\geng{h_3,\dots,h_n}{s_3,\dots,s_n}{\Lambda_3,\dots,\Lambda_{n-1}}.
	\end{align*}
	Continuing this process $n-2$ times, we find that $(\ref{eq:capCsChsh})=\langle 1 \rangle$.
	Thus, $f=g$, a contradiction. We conclude that the map is faithful.

	By applying the Orbit-Stabilizer Theorem, Theorem \ref{thm:|CG(s1)|} along with equation (\ref{eq:orderCn}), we have
	\begin{align*}
		|V| = \frac{|\cC_n|}{|\cent_{\cC_n}(s_1)|} =2 (4^n-1).
	\end{align*}
	The proof is now complete.
\end{proof}

\begin{corollary}\label{cor:center}
	The center of $\cC_n$ is $\langle 1\rangle$.
\end{corollary}

As stated in Section \ref{sec:intro}, the size of the conjugacy class of $S_1$ in ${\rm Cliff}_n$ matches the size of the conjugacy class of $s_1$ in $\cC_n$. The Algorithm \ref{alg:construct_Cn} provides an efficient construction of $\cC_n$ using matrix multiplication, avoiding the need to compute a quotient.
\begin{center}
	\begin{mdframed}[roundcorner=10pt, linecolor=black, linewidth=1pt, backgroundcolor=white, innerleftmargin=10pt, innerrightmargin=10pt] 
		\begin{algorithm}[H]
			\caption{Constructing the $\cC_n$}
			\label{alg:construct_Cn}
			
			\SetKwInOut{Input}{Input}
			\SetKwInOut{Output}{Output}
			
			\Input{$n$}
			\Output{Permutation form of the generators of $\cC_n$}
			
			\BlankLine
			Construct the set of generators $G= \{H_i, S_i, CZ_j \}$\;
			
			\BlankLine
			Initialize the set $L \gets \{S_1\}$\;
			\While{$|L| < 2(4^n - 1)$}{
				\ForEach{$U \in G$}{
					$L \gets L \cup \{UCU^{-1}: \text{for }C\in L\}$\;
				}
			}
			
			\BlankLine
			Convert $L$ to a list\;
			Represent conjugate actions of $H_i$, $S_i$, and $CZ_j$ on $L$ as permutations\;
			
			\BlankLine
			\Return{Permutation form of the generators of $\cC_n$}\;
		\end{algorithm}
	\end{mdframed}
\end{center}

\subsection{Application: A presentations of $\cI\cN_n$}\label{sec:presentofINn}

Recall the normal subgroups of $\cC_n$:
\begin{theorem}\cite[Section 5]{LYPL} The groups $\cP_2$ and $[\cC_2, \cC_2]$ are the only non-trivial proper normal subgroups of $\cC_2$. For $n \geq 3$, $\cP_n$ is the only non-trivial proper normal subgroup of $\cC_n$. 
\end{theorem}
The Clifford theory examines the representation theory of a group by analyzing the representations of its normal subgroups.
In this context, the inertia subgroups of the irreducible representations of normal subgroups play a significant role.
Mastel \cite{Ma} investigated the Clifford theory for $\cC_n$ and provided the following lemma:
\begin{lemma}\cite[Lemma IV.1]{Ma}
	All non-trivial irreducible representations of $\cP_n$ are conjugate in $\cC_n$. 
\end{lemma}
Combining this with the fact that conjugate representations of normal subgroups have isomorphic inertia subgroups, we conclude that there is only one inertia subgroup for the non-trivial representations of $\cP_n$ in $\cC_n$, which is described in the following theorem:
\begin{theorem} \cite[Theorem IV.3]{Ma}
	The inertia subgroup of a non-trivial representation of \( \cP_n \) in \( \cC_n \) is isomorphic to 
	\begin{align*}
		\cI\cN_n:= \left\{\begin{array}{ll}
			\langle h_1,h_2,s_2,h_1s_1^2h_1,M \rangle,&\text{if }n=2,\\
			\geng{h_1,h_2,\dots, h_n}{s_2,\dots,s_n}{ M, \Lambda_{2},\dots, \Lambda_{n-1}},&\text{if }n\geq 3, 
		\end{array}\right.
	\end{align*}
	where $M=  h_2\Lambda_1s_2^2h_2h_1s_1^2 h_2\Lambda_1h_2 $.
\end{theorem}
\begin{proof} 	
	In the statement of \cite[Theorem IV.3]{Ma}, an additional generator $h_1s_1^2h_1$ is included for $n\geq 3$.
	It follows from Lemma \ref{lem:tech_2} that 
	\begin{align*}
		h_1s_1^2h_1 \in \left\langle
		\begin{matrix}
			h_1,h_2,h_3 ,\\
			M,\Lambda_2
		\end{matrix}\right\rangle.
	\end{align*}
	Thus the generator $h_1s_1^2h_1$ becomes redundant for $n\geq 3$.
\end{proof}
The order of $\cI\cN_n$ is
\begin{align}\label{eq:orderINn}
	|\cI\cN_n| = 2^{n^2+2n} \prod_{i=1}^{n-1} (2^{2i} - 1).
\end{align}

The irreducible representations of $\cC_n$ are determined by $\cI\cN_n / \cP_n$ and ${\rm Sp}(2n, 2)$, as shown in \cite[Theorem IV.4]{Ma}.

\begin{theorem}
	The subgroups $\cent_{\cC_n}(z_1)$ and $\cI\cN_n$ are conjugate within $\cC_n$.
\end{theorem}
\begin{proof}
It follows from Lemma \ref{lem:tech_2} that 
\begin{align*}
\geng{h_1,h_2}{s_2}{s_1^3h_1\Lambda_1h_1s_1}&= \geng{h_1,h_2}{s_2}{h_1s_1^2h_1,M},\\
\geng{h_2,h_3}{s_2,s_3}{s_1^3h_1\Lambda_1h_1s_1,s_1^3h_1\Lambda_2h_1s_1} &=\geng{h_2,h_3}{s_2,s_3}{M,\Lambda_2}.
\end{align*}
Combined with Theorem \ref{thm:CG(s1^2)}, we have 
\begin{align*}
s_1^3h_1\cent_{\cC_2}(s_1^2)h_1s_1&=\geng{s_1^3h_1h_1s_1h_1s_1^3h_1h_1s_1,h_2}{s_2}{s_1^3h_1\Lambda_1h_1s_1} =\geng{h_1,h_2}{s_2}{h_1s_1^2h_1,M}=\cI\cN_2
\end{align*}
and  
	\begin{align*}
		s_1^3h_1\cent_{\cC_n}(s_1^2)h_1s_1&=\geng{s_1^3h_1h_1s_1h_1s_1^3h_1h_1s_1,h_2,\dots,h_n}{s_2,\dots,s_n}{s_1^3h_1\Lambda_1h_1s_1,s_1^3h_1\Lambda_2h_1s_1,\dots,\Lambda_{n-1}}\\
		&=\geng{h_1,h_2,\dots,h_n}{s_2,\dots,s_n}{M,\Lambda_2,\dots,\Lambda_{n-1}}= \cI\cN_n.
	\end{align*}
for $n\geq 3$.
The proof is complete.
\end{proof}

Let $n\geq 2$ and $g=h_1s_1h_1s_1^3h_1$. Recall that 
\begin{align*}
	\cent_{\cC_n}(s_1)=\geng{h_2,\dots,h_n}{s_1,\dots,s_n}{\Lambda_1,\dots,\Lambda_{n-1}} \trianglelefteq \geng{g,h_2,\dots,h_n}{s_1,\dots,s_n}{\Lambda_1,\dots,\Lambda_{n-1}}=\cent_{\cC_n}(s_1^2) \simeq 	\cI\cN_n .
\end{align*}
Observe that 
\begin{align*}
	\cent_{\cC_n}(s_1^2) =  \langle g\rangle  \ltimes_\phi \cent_{\cC_n}(s_1) 
\end{align*}
where $\phi:\langle g\rangle \to {\rm Aut}(\cent_{\cC_n}(s_1))$ is the homomorphism given by $\phi(u)(v)=u^{-1}v u$ for $u\in \langle g\rangle$ and $v\in \cent_{\cC_n}(s_1)$. 

\begin{theorem}
	Let $n\geq 2$.
	The group $\cI\cN_n$ is generated by the symbols 
	$$\{g\}\cup \{ h_i, s_j, \Lambda_{k} : 2\leq i\leq n, \, 1\leq j\leq n, \, 1\leq k\leq n-1\}$$ subject to the following relations:
	\begin{itemize}
		\item(Q1) $g^2=1$.\\[-3mm]
		\item(Q2) $gh_ig=h_i$ for all $i$.\\[-3mm]
		\item(Q3) $gs_1g=s_1^3$ and $gs_jg=s_j$ for all $j\neq 1$.\\[-3mm]
		\item(Q4) $g\Lambda_1g=\Lambda_1s_2^2$ and $g\Lambda_kg=\Lambda_k$ for all $k\neq 1$.\\[-3mm]
		\item(R1) $h_i^2=s_j^4=1$ for all $i,j$. \\[-3mm]
		\item(R2) $(s_ih_i)^3=1$ for all $i$.  \\[-3mm]

		\item(R3) $\Lambda_k^2=1$ for all $k$.\\[-3mm]
		\item(R4) $h_ih_{j}=h_{j}h_i$, \,\, $s_{i}s_{j}=s_{j}s_i$, \,\, and \,\, $h_is_j=s_jh_i$ for all $i,j$ with $i\neq j$. \\[-3mm]		
		\item(R5) $h_{i}\Lambda_k=\Lambda_kh_{i}$ for all $i\notin\{k,k+1\}$.\\[-3mm]
		\item(R6) $s_j\Lambda_k=\Lambda_ks_j$ for all $j,k$. \\[-3mm]
		
		\item(R7) $h_is_i^2h_i\Lambda_i=\Lambda_ih_is_i^2h_is_{i+1}^2$ for all $2\leq i\leq n-1$.\\[-3mm]
		\item(R8) $h_{i+1}s_{i+1}^2h_{i+1}\Lambda_i=\Lambda_is_i^2h_{i+1}s_{i+1}^2h_{i+1}$ for all $1\leq i\leq n-1$.\\[-3mm]
		\item(R9) $\Lambda_ih_i\Lambda_i =s_ih_i\Lambda_is_{i}h_is_is_{i+1} $  for all $2\leq i\leq n-1$.\\[-3mm]
		\item(R10) $\Lambda_ih_{i+1}\Lambda_i =s_{i+1}h_{i+1}\Lambda_is_{i}s_{i+1}h_{i+1}s_{i+1}$ for all $1\leq i\leq n-1$.\\[1mm]
		When $n\geq 3$, we also have
		\item(R11) $\Lambda_k\Lambda_{k'}=\Lambda_{k'}\Lambda_k$ for all $k\neq k'$.\\[-3mm]
		
		\item(R12) $(\Lambda_{i+1}h_{i+1}h_{i+2}\Lambda_{i+1}h_{i+1}h_{i+2}\Lambda_{i})^3=1$ for all $1\leq i\leq n-2$.\\[1mm]
		When $n\geq 4$, we also have
		\item(R13) $(\Lambda_{i}h_{i}h_{i+1}\Lambda_ih_{i}h_{i+1}\Lambda_{i+1})^3=1$ for all $2\leq i\leq n-2$.\\[-3mm]
		\item(R14) $\Lambda_ih_{i}h_{i+1}\Lambda_ih_{i+1}h_{i+2}\Lambda_{i+1}h_{i+1}h_{i+2}\Lambda_ih_{i}h_{i+1}\Lambda_i=\\[1mm]
		\Lambda_{i+1}h_{i+1}h_{i+2}\Lambda_{i+1}h_{i}h_{i+1}\Lambda_ih_{i}h_{i+1}\Lambda_{i+1}h_{i+1}h_{i+2}\Lambda_{i+1}$ for all $2\leq i\leq n-2$.
	\end{itemize}	
	
\end{theorem}
\begin{proof}
	Let $\cG_n$ be the group generated by the symbols 
	$$ \{ h_i, s_j, \Lambda_{k} : 2\leq i\leq n, \, 1\leq j\leq n, \, 1\leq k\leq n-1\}$$	
	subject to the relations (R1)--(R$r$), where $r = 10$ if $n = 2$, $r = 12$ if $n = 3$, and $r = 14$ if $n \geq 4$.
	By Theorem \ref{thm:CG(s1)}, the map $\cG_n \to \cent_{\cC_n}(s_1)$, which sends $h_i \mapsto h_i$, $s_j \mapsto s_j$, and $\Lambda_k \mapsto \Lambda_k$, is a surjective homomorphism.
	On the other hand, the relations (R1)--(R$r$), which correspond to the relations in \cite[Figure 8]{Sel}, can derive all the rewriting rules required to transform $\nf(u)\nf(v)$ into $\nf(uv)$ for all $u,v \in \cent_{\cC_n}(s_1)$.
	This induces a surjective homomorphism from $\cent_{\cC_n}(s_1)$ onto $\cG_n$.
	Hence, $\cG_n$ is isomorphic to $\cent_{\cC_n}(s_1)$.
	
	Observe that 
	\begin{align*}
		\phi(g)(s_1)&= s_1^{3},\\
		\phi(g)(\Lambda_1)&=\Lambda_1s_2^2,\\
		\phi(g)(v) &=v\quad\text{ for all }v\in \{h_2,\dots,h_n,s_2,\dots,s_n,\Lambda_2,\dots,\Lambda_{n-1}\}.
	\end{align*}
	Applying the method in \cite[Section 10]{Joh} yields that   $\langle g\rangle\ltimes_\phi	\cent_{\cC_n}(s_1) $ has a presentation  
	\begin{align*}
		\langle g,V  \,|\, g^2=1, g^{-1}vg=\phi(g)(v) \text{ for }v\in V,\text{(R1)--(R$r$)} \rangle.
	\end{align*}
	This completes the proof.
\end{proof}

\section{Conclusions}\label{sec:openprob}
In this paper, we determined the centralizers of the $z_1$ and
$s_1$ gates in $\cC_n$ using the normal form. As applications,
we obtained a faithful permutation representation of $\cC_n$ of
degree $2(4^n-1)$ and a presentation of the inertia subgroup
$\cI\cN_n$.

The same normal form approach also yields explicit descriptions of
the centralizers of $h_1$ and $\Lambda_1$. Their proofs are omitted,
as they require lengthy case analyses beyond the scope of the present
paper.
Complete proofs can be found in \cite{LeeThesis}.

\begin{theorem*}\label{thm:CG(h1)} The centralizer of $h_1$ in $\cC_n$ is given by
	\begin{align*}
		\cent_{\cC_n}(h_1) =\left\{\begin{array}{ll}
			\langle h_1s_1^2h_1s_1^2,h_1\rangle,&\text{if }n=1,\\
			\langle h_1s_1^2h_1s_1^2,h_1,h_2,s_2\rangle,&\text{if }n=2,\\
			\geng{h_1s_1^2h_1s_1^2,h_1,\dots,h_n}{s_2, \dots, s_n}{ \Lambda_{2}, \dots, \Lambda_{n-1 }},&\text{if }n\geq 3.
		\end{array}
		\right. 
	\end{align*}
	Moreover, the order of $\cent_{\cC_n}(h_1)$ is 
	\begin{align*}
		|\cent_{\cC_n}(h_1)| =\begin{cases}
			4,&\text{if }n=1,\\
			4|\cC_{n-1}|,&\text{if }n\geq 2.
		\end{cases}
	\end{align*}
\end{theorem*}
Let $v_1=s_1s_2^2h_2s_2^2\Lambda_1h_2\Lambda_1$, $v_2=h_1h_2s_2h_2s_2\Lambda_1 h_1h_2\Lambda_1h_1h_2$ and 
\begin{equation*}
	\scalebox{0.55}{
		\begin{quantikz}
			&\gate[3]{g_{i,j,k}}&\qw\\
			&                   &\qw\\
			&                   &\qw
		\end{quantikz}
	}:=
	\scalebox{0.55}{
		\begin{quantikz}
			&\qw        &\gate[2]{B_j}&\gate{C_k}&\gate[2]{D_1}&\qw          &\qw\\
			&\gate[2]{B_i}&            &\qw     &             &\gate[2]{D_1}&\qw\\
			&           &\qw         &\qw     &\qw          &             &\qw
		\end{quantikz}
	}
\end{equation*}
for all $i,j\in\{1,4\}$ and $k\in\{1,2\}$.
\begin{theorem*}
	The centralizer of $\Lambda_1$ in $\cC_n$ is given by
	\begin{align*}
		\cent_{\cC_n}(\Lambda_1) =\begin{cases}
			\displaystyle{  \geng{v_1,v_2}  {s_1,s_2}{\Lambda_1}},&\text{if }n=2,\\
			\displaystyle{  \geng{g_{1,1,1},\dots,g_{4,4,2},v_1,v_2 ,h_3,\dots,h_n }{s_1,s_2,s_3,\dots,s_n}{\Lambda_1,\dots,\Lambda_{n-1}}},&\text{if }n\geq 3.			
		\end{cases}
	\end{align*}
	Moreover, the order of $\cent_{\cC_n}(\Lambda_1)$ is 
	\begin{align*}
		|\cent_{\cC_n}(\Lambda_1)|=\begin{cases}
			192,&\text{if }n=2,\\
			\displaystyle{3\cdot 2^{n^2+2n-2}\prod_{i=1}^{n-2}(2^{2i}-1) },&\text{if }n\geq 3.
		\end{cases}
	\end{align*}
\end{theorem*}

\section*{Acknowledgments}
I would like to thank Prof. Yung-Ning Peng for his helpful comments and discussions, as well as the anonymous reviewers for their valuable comments and suggestions that improved the presentation of this paper.

\appendix

\section{A presentation of projective Clifford groups}\label{sec:appendix}

\begin{theorem}\label{thm:presentationCn}
	The group $\cC_n$ is generated by the symbols
	\begin{align*}
		\begin{cases}
			\{ h_1, s_1\},&\text{if }n=1,\\
			\{ h_1,\dots, h_n,s_1,\dots,s_n,\Lambda_{1},\dots,\Lambda_{n-1}\},&\text{if }n\geq 2
		\end{cases}
	\end{align*}
	subject to the following relations:	
	\begin{itemize}
		\item(R1) $h_i^2=s_i^4=1$ for all $i$. \\[-3mm]
		\item(R2) $(s_ih_i)^3=1$ for all $i$.  \\[1mm]		
		When $n\geq 2$:
		
		\item(R3) $\Lambda_j^2=1$ for all $j$.\\[-3mm]
		\item(R4) $h_ih_{j}=h_{j}h_i$, \,\, $s_{i}s_{j}=s_{j}s_i$, \,\, and \,\, $h_is_j=s_jh_i$ for all $i\neq j$. \\[-3mm]		
		\item(R5) $h_{i}\Lambda_j=\Lambda_jh_{i}$ for all $i\notin\{j,j+1\}$.\\[-3mm]
		\item(R6) $s_i\Lambda_j=\Lambda_js_i$ for all $i, j$. \\[-3mm]
		
		\item(R7) $h_is_i^2h_i\Lambda_i=\Lambda_ih_is_i^2h_is_{i+1}^2$ for all $1\leq i\leq n-1$.\\[-3mm]
		\item(R8) $h_{i+1}s_{i+1}^2h_{i+1}\Lambda_i=\Lambda_is_i^2h_{i+1}s_{i+1}^2h_{i+1}$ for all $1\leq i\leq n-1$.\\[-3mm]
		\item(R9) $\Lambda_ih_i\Lambda_i =s_ih_i\Lambda_is_{i}h_is_is_{i+1} $  for all $1\leq i\leq n-1$.\\[-3mm]
		\item(R10) $\Lambda_ih_{i+1}\Lambda_i =s_{i+1}h_{i+1}\Lambda_is_{i}s_{i+1}h_{i+1}s_{i+1}$ for all $1\leq i\leq n-1$.\\[1mm]
		When $n\geq 3$, we also have
		\item(R11) $\Lambda_i\Lambda_j=\Lambda_j\Lambda_i$ for all $i\neq j$.\\[-3mm]
		\item(R12) $\Lambda_ih_{i}h_{i+1}\Lambda_ih_{i+1}h_{i+2}\Lambda_{i+1}h_{i+1}h_{i+2}\Lambda_ih_{i}h_{i+1}\Lambda_i=\\[1mm]
		\Lambda_{i+1}h_{i+1}h_{i+2}\Lambda_{i+1}h_{i}h_{i+1}\Lambda_ih_{i}h_{i+1}\Lambda_{i+1}h_{i+1}h_{i+2}\Lambda_{i+1}$ for all $1\leq i\leq n-2$.
		\item(R13) $(\Lambda_{i}h_{i}h_{i+1}\Lambda_ih_{i}h_{i+1}\Lambda_{i+1})^3=1$ for all $1\leq i\leq n-2$.\\[-3mm]
		\item(R14) $(\Lambda_{i+1}h_{i+1}h_{i+2}\Lambda_{i+1}h_{i+1}h_{i+2}\Lambda_{i})^3=1$ for all $1\leq i\leq n-2$.
	\end{itemize}
\end{theorem}
\begin{proof}
	Let $\cG_n$ denote the group defined in the theorem statement.
	It is easy to see that the map $\cG_n\to \cC_n$ that send $h_i,s_i,\Lambda_j$ to $h_i,s_i,\Lambda_j$ is a surjective homomorphism.
	On the other hand, the relations (R1)--(R$r$) ($r=2$ if $n=1$, $r=10$ if $n=2$, $r=14$ if $n\geq 3$), which correspond to the relations in \cite[Figure 8]{Sel}, can derive all the rewriting rules required to transform $\nf(f)\nf(g)$ into $\nf(fg)$ for all $f,g \in \cC_n$.
	This induces a surjective homomorphism from $\cC_n$ onto $\cG_n$.
	Hence, $\cG_n$ is isomorphic to $\cC_n$.
\end{proof}

\end{document}